\documentclass{amsart}

\usepackage[utf8x]{inputenc}
\usepackage{geometry}
\usepackage{paralist}
\usepackage{bm}
\usepackage{bbm}
\usepackage{graphics}
\usepackage[dvipdf]{graphicx}

\theoremstyle{definition}
\newtheorem{thrm}{Theorem}
\newtheorem{dfntn}{Definition}
\newtheorem{lmm}{Lemma}

\theoremstyle{plain}

\newtheorem{cond}[thrm]{Condition}
\newtheorem{rmrk}[thrm]{Remark}

\numberwithin{equation}{section}

\newcommand{\N}{{\mathbb{N}}}
\newcommand\R{\mathbb{R}}
\newcommand\Ru{{\mathbb{R} \cup \left\{\infty\right\}}}
\newcommand\Rd{{\mathbb{R}^d}}
\newcommand{\X}{{\bm{X}}}
\newcommand{\bd}{{\bm{d}}}
\newcommand{\Sd}{{\mathbb{S}^d}}

\newcommand{\prb}{\mathcal{P}_2(\Omega)}
\newcommand{\wass}{\mathbf{W}_2}
\newcommand{\pgen}{\mathcal{P}_2}
\newcommand{\bmu}{\boldsymbol{\mu}}

\newcommand{\mf}{\mathcal{M}}
\newcommand{\g}{\mathbf{g}}
\newcommand{\smf}{\mathcal{N}}
\newcommand{\sg}{\mathbf{h}}
\newcommand{\cut}{\operatorname{cut}}
\newcommand{\hess}{\operatorname{Hess}}
\newcommand{\scurv}{\mathfrak{S}}

\newcommand{\minus}{\scalebox{0.4}[0.9]{$-$}}

\newcommand{\dd}{\,\mathrm{d}}

\newcommand\ninN{{n\in\mathbb{N}}}
\newcommand\kinN{{k\in\mathbb{N}}}
\newcommand\ntoinf{{n\rightarrow \infty}}
\newcommand\dar{{\ \stackrel{\bd}{\rightarrow}} \ }

\newcommand\uk{u^k_{\tau}}
\newcommand\ukm{u^{k-1}_{\tau}}
\newcommand\ukzm{u^{k-2}_{\tau}}

\newcommand{\dom}{\mathcal{D}(\mathcal{E})}
\newcommand{\nrg}{\mathcal{E}}

\renewcommand{\tilde}{\widetilde}
\renewcommand{\hat}{\widehat}
\renewcommand{\bar}{\overline}

\begin{document}
\title{A Variational Formulation of the BDF2 Method for Metric Gradient Flows}\thanks{This research was supported by the DFG Collaborative Research Center TRR 109, `Discretization in Geometry and Dynamics'.}
\author{Daniel Matthes}\address{Zentrum f\"ur Mathematik, Technische Universit\"at M\"unchen, 85747 Garching, Germany}
\author{Simon Plazotta}\address{Zentrum f\"ur Mathematik, Technische Universit\"at M\"unchen, 85747 Garching, Germany}
\date{\today}
\begin{abstract} 
  We propose a variational form of the BDF2 method as an alternative to the commonly used minimizing movement scheme for the time-discrete approximation of gradient flows in abstract metric spaces.
   Assuming uniform semi-convexity --- but no smoothness --- of the augmented energy functional, we prove well-posedness of the method and convergence of the discrete approximations to a curve of steepest descent.
    In a smooth Hilbertian setting, classical theory would predict a convergence order of two in time, we prove convergence order of one-half in the general metric setting and under our weak hypotheses. 
    Further, we illustrate these results with numerical experiments for gradient flows on a compact Riemannian manifold, in a Hilbert space, and in the $L^2$-Wasserstein metric.
\end{abstract}
\maketitle
\section{Introduction}\label{sec:Intro}
\subsection{The Main Idea in Short}
\label{sct:introtrivial}
This article is concerned with a particular temporal discretization for gradient flows in metric spaces.
That is, we study the approximation of curves of steepest descent 
in the energy landscape of a functional $\nrg:\X\to\Ru$ with respect to a metric $\bd$ on $\X$.
Before we elaborate on our motivation and results,
we briefly outline the concept in the simplest setting, 
namely when $\X=\Rd$, $\bd$ is the Euclidean metric, and $\nrg\in C^\infty(\Rd)$, 
in which case the problem amounts to approximate solutions to
\begin{align}
  \label{eq:simple}
  \dot u = -\nabla\nrg(u).
\end{align}
Assuming in addition that $\nrg$ is uniformly semi-convex, 
i.e., $\nabla^2\nrg\ge\lambda\mathbf{1}_d$ for some $\lambda\in\R$,
then it follows that the implicit Euler method with any sufficiently small uniform time step $\tau>0$,
\begin{align}
  \label{eq:bdf1}
  \frac{u_\tau^k-u_\tau^{k-1}}{\tau} = -\nabla\nrg(u_\tau^k),
\end{align}
is well-defined, i.e., the initial condition $u_\tau^0$ determines the entire sequence $(u_\tau^k)_{k\in\N}$ uniquely.
It is further well-known that this is a first order approximation of the true solution $u$ to \eqref{eq:simple},
i.e., $u_\tau^k=u(k\tau)+O(\tau)$ as $\tau\to0$.
Our point is that under the same semi-convexity hypothesis,
also the second order Backward Differentiation Formula (BDF2) method, 
\begin{align}
  \label{eq:bdf2}
  \frac{3u_\tau^k-4u_\tau^{k-1}+u_\tau^{k-2}}{2\tau} = -\nabla\nrg(u_\tau^k),
\end{align}
is well-defined and convergent.
The strength of the BDF2 method in comparison to the implicit Euler scheme is that the former
--- at least in the smooth setting at hand --- converges to second order in $\tau$. \\

Studies on the BDF2 scheme in the above mentioned ODE setting have been an active topic in the 1950's and 1960's \cite{dahlquist1956convergence,dahlquist1963special}.
Subsequently, the method in the Hilbertian setup, where $\X$ is a Hilbert space and $\bd$ is induced by the norm, 
has attracted a lot of attention.
Convergence results, particularly for very general nonlinear right-hand sides, are much more recent, 
see e.g.
\cite{akrivis2004linearly,akrivis2015fully,Baiocchi,bramble1989incomplete,emmrich2005stability,emmrich2009convergence,emmrich2009two,hansen2006convergence,kreth1978time,le1979semidiscretization,Thomee}.
Still, the analysis appears to be more or less complete now, at least under reasonable conditions on the nonlinearity. On the other hand,
we are apparently the first to analyze (a variational formulation of) the BDF2 method
\emph{for approximation of gradient flows in abstract metric spaces},
and to prove its convergence just under the hypothesis of semi-convexity.
Our proof is different from the one in ODE text books \cite{BD,Gear,Hairer2013},
also from the ones typically given in the Hilbertian setting, like in \cite{emmrich2005stability}.
The key difference is that due to the possible ``roughness'' of the metric space $\X$,
there is no appropriate notion of smooth solution for the gradient flow
(in general, there does not even exist a good definition for the differentiability of a curve).
Hence, we cannot invoke error estimates that rely on Taylor expansions around the limiting solution.
Instead of applying the usual error estimates between discrete solutions for the same time step $\tau$, 
we need to resort to Cauchy-type estimates for solutions at different time steps as in \cite[Section 4.1]{ags}. 
This yields a control on the global approximation error, which is not of order two, but only of order one-half.
We also provide an example to show that indeed, even for specific, seemingly harmless choices of $(\X,\bd)$ and $\nrg$,
convergence takes place at first order only. 
On the other hand, in view of the results in \cite[Section 4.4]{ags} on the implicit Euler method,
it seems likely that our variational BDF2 converges to first order in general.
Currently, we are not able to close the apparent gap between order one-half and order one,
mainly because the calculations for BDF2 are much more complex than the corresponding ones in \cite[Section 4.4]{ags}.
And according to our general philosophy, that we describe below, 
any further investigations in the direction of improving the rate beyond one-half appear rather pointless.

We emphasize that the proven slow convergence order one-half does not contradict our initial intention 
of providing a method of faster convergence than the implicit Euler one.
Indeed, \emph{if} the approximated solution is smooth enough 
(which, in specific situations, can often be verified a posteriori by considering it in a different setting),
\emph{then} the classical convergence proofs from text books apply and yield the desired rate of order two.
That philosophy is justified by a series of numerical experiments that all show second order convergence.
\emph{Our contribution is that 
  --- regardless of the regularity of the limiting solution under consideration ---
  convergence of the method is guaranteed, even with an explicit rate.
  And our proof utilizes solely the variational structure of the scheme \eqref{eq:bdf2} and the semi-convexity hypothesis on $\nrg$.
}

We remark that other, conceptually different approaches to construct time discretizations with formally higher order have been investigated, for instance, variational formulations of Runge-Kutta methods \cite{LEGENDRE2017345,laguzet:hal-01404619}. 
However, there is no analytically proven rate of convergence available so far, not even order one-half. 

\subsection{Metric Gradient Flows}
For definiteness, we are working inside the abstract framework developed in the first part of the book \cite{ags};
only a few basic notations from that comprehensive theory will be relevant to us, 
and these are summarized in Section \ref{sct:ags} below. 
Although our considerations are very general, we have three specific settings in mind.
The first is that of gradient flows for smooth functions on a finite dimensional compact manifold,
the second concerns uniformly semi-convex functionals on Hilbert spaces,
and in the third, we consider flows for uniformly displacement semi-convex functionals
on the space $\X=\prb$ of probability measures with respect to the $L^2$-Wasserstein distance $\bd=\wass$.

Particularly the third setting has gained a lot of popularity in the past two decades, see \cite{santambrogio2015optimal,villani,villani2008optimal} for a general introduction to the $L^2$-Wasserstein space $(\prb,\wass)$ and see \cite{MCCANN1997153} for the concept of displacement convexity.
It has been shown that a variety of important nonlinear dissipative evolution equations 
can be cast in the form of a Wasserstein gradient flow,
from linear, nonlinear, and non-local Fokker-Planck equations \cite{carrillo2011,jko,otto2001}
over fourth order fluid and quantum models \cite{GOthinfilm,gianazza2009,MatthesMcCannSavare} 
to chemotaxis systems \cite{BCCkellersegel,blanchet2012,zinsl2015exponential}.
The $L^2$-Wasserstein framework has been extended in many different directions,
that allow to consider 
reaction diffusion equations \cite{Glitzky,Mielke}, Poisson-Nernst-Planck equations \cite{Kinderlehrer2016}, 
multi-component fluid systems \cite{laurencot2011},
and Cahn-Hilliard equations \cite{lisini2012},
as well as Markov chains \cite{MaasMC,MielkeMC}
as gradient flows in suitable metrics. \\

We shall not discuss the manifold advantages of such a variational formulation of the evolution,
but focus only on the fact that under certain convexity properties of $\nrg$, 
a gradient flow can be rather easily obtained as the limit of a convergent time-discrete approximation,
for instance by the variational forms of the implicit Euler and BDF2 methods discussed below.
While the order of convergence of the time discretization is usually irrelevant for existence proofs
(and the simpler Euler scheme might be advantageous over the more complicated BDF2 method),
it is a limiting factor for efficiency of numerical discretizations that are based on these approximations.
Indeed, a variety of structure preserving full (spatio-temporal) discretizations 
for numerical solution of Wasserstein gradient flows that have been developed recently,
see e.g. \cite{BCMO,JMO,Peyre}, and essentially all of these methods discretize by implicit Euler in time direction.
It seems that the variational character of the Euler scheme is essential 
to inherit desirable properties --- like monotonicity of the energy --- from the original flow 
to its discrete approximation,
and is even the key element for convergence proofs \cite{MO1}.
\emph{One of the conclusions from our work is that the variational BDF2 method introduced below
might give a better approximation in time, 
without losing convergence or the flow's variational character.}

\subsection{Minimizing Movement Scheme}\label{subsec:MMS}
Before we introduce our own temporal discretization, 
we recall some properties of the celebrated minimizing movement scheme,
also known as implicit Euler method or JKO stepping, depending on the context. Let the metric space $(\X,\bd)$ and the functional $\nrg$ be fixed;
we wish to construct a curve of steepest descent emerging from the initial data $u_0\in\X$. \\

\noindent  For metric gradient flows in the sense of \cite{ags,de1993new}, that scheme is defined as follows. 
\begin{quote}
  For each sufficiently small time step $\tau>0$, 
  let an initial condition $u_\tau^0$ be given that approximates $u^0$.
  Then define inductively a discrete solution $(u_\tau^k)_{k\in\N}$ such that
  each $u_\tau^k$ with $k=1,2,\ldots$ \,is a minimizer of
  the Yosida-penalized energy functional
  \begin{align*}
    w\mapsto\Phi(\tau,u_\tau^{k-1};w) := \frac1{2\tau}\bd^2(u_\tau^{k-1},w) + \nrg(w).
  \end{align*}
\end{quote}
There are various ``soft'' conditions that guarantee well-definedness of this scheme, 
i.e., the inductive solvability of the minimization problems;
one is formulated as Condition \ref{cond:ags} below.
It is easy to verify that, in the trivial setting discussed in Section \ref{sct:introtrivial},
the Euler-Lagrange equation for minimization of $\Phi(\tau,u_\tau^{k-1};\cdot)$ 
yields precisely the induction formula \eqref{eq:bdf1} for the implicit Euler scheme. 

One of the remarkable strengths of this method in the abstract setting is its stability:
by an easy induction argument, one proves that for arbitrary $0\le m\le n$,
\begin{align}
  \label{eq:stable1}
  \nrg(u_\tau^n) + \frac\tau{2}\sum_{k=m+1}^n\left(\frac{\bd(u_\tau^{k-1},u_\tau^{k})}\tau\right)^2 \le \nrg(u_\tau^m),
\end{align}
which immediately implies monotonicity of the energy, $\nrg(u_\tau^n)\le\nrg(u_\tau^m)$,
as well as a $\tau$-uniform bound on the ``integrated kinetic energy''.
These bounds are usually sufficient to conclude the convergence of the discrete solution
to a continuous curve $u_*$ in the limit $\tau\searrow0$.
Some additional work is needed to prove that $u_*$ is indeed a curve of steepest descent.
A famous hypothesis that makes this last step work is the convexity condition \cite[Assumption 4.0.4]{ags}:
\begin{cond}
  \label{cond:ags}
  There exists a $\lambda\in\R$ such that 
  for each sufficiently small $\tau>0$, for each reference point $u\in\dom$, 
  and for each pair of points $\gamma_0,\gamma_1\in\X$,
  there is some curve $(\gamma_s)_{s\in[0,1]}$ connecting $\gamma_0$ to $\gamma_1$ 
  along which $s\mapsto\Phi(\tau,u;\gamma_s)$ is uniformly convex of modulus $\frac1\tau+\lambda$.  
\end{cond}
We remark that recent adaptations of the minimizing movement scheme, see e.g. \cite{fg,pz,rms},
allow to treat also non-autonomous gradient flows along the same lines,
as long as the time-dependent energy functional has some ``convexity in the average''.

\subsection{BDF2 Method}
We propose the following construction of a discrete approximation $(u_\tau^k)_{k\in\N}$:
\begin{quote}
  For each sufficiently small time step $\tau>0$, 
  let a pair of initial conditions $(u_\tau^{\minus 1},u_\tau^{0})$ be given that approximate $u^0$.
  Then define inductively a discrete solution $(u_\tau^k)_{k\in\N}$ such that
  each $u_\tau^k$ with $k\in\N$ is a minimizer of
  the following functional,
  \begin{align*}
    w\mapsto\Psi(\tau,u_\tau^{k-2},u_\tau^{k-1};w) 
    := \frac1\tau\bd^2(w,u_\tau^{k-1})-\frac1{4\tau}\bd^2(w,u_\tau^{k-2})+\nrg(w).
  \end{align*}
\end{quote}
In the euclidean setting of Section \ref{sct:introtrivial},
the minimizer $u_\tau^k$ satisfies the BDF2 recursion \eqref{eq:bdf2}.
The BDF2 method is known to converge to second order under suitable smoothness hypotheses, see e.g. \cite{BD}.
A key feature of this approach is that a simple induction argument produces 
an intrinsic stability estimate that is similar to --- although slightly weaker than --- \eqref{eq:stable1} above:
\begin{align}
  \label{eq:stable2}
  \nrg(u_\tau^n) + \frac\tau{4}\sum_{k=m+1}^n\left(\frac{\bd(u_\tau^{k-1},u_\tau^{k})}\tau\right)^2 
  \le \nrg(u_\tau^m) + \frac1{4\tau}\bd^2(u_\tau^{m-1},u_\tau^{m}).
\end{align}
From \eqref{eq:stable2} for $n=m+1$, 
it is clear that the energy value $\nrg(u_\tau^k)$ is not necessarily diminished in each time step,
but the potential loss of monotonicity is well controlled.
Moreover, one still has the crucial bound on the ``integrated kinetic energy''.
The validity of estimate \eqref{eq:stable2} is related to the intrinsic \emph{A-stability} of the BDF2 scheme \eqref{eq:bdf2}.
Since the implicit Euler (BDF1) and the BDF2 methods are \emph{the only} A-stable backward differentiation formulae,
it is not to be expected that one can invent a variant of these schemes which a consistency order higher than two,
which also satisfies a stability estimate like \eqref{eq:bdf1}.
In this sense, our approach is optimal.

\subsection{Main Result}
Our main result, stated in Theorem \ref{thm:mainthm}, can be informally rephrased as follows.
We assume that $\nrg$ is lower semi-continuous, admits a sort of lower bound,
and satisfies the following adaptation of Condition \ref{cond:ags}:
\begin{cond}
  \label{cond:bdf2}
  There exists a $\lambda\in\R$ such that 
  for each sufficiently small $\tau>0$, for each pair of reference points $u,v\in\dom$, 
  and for each pair of points $\gamma_0,\gamma_1\in\X$,
  there is some curve $(\gamma_s)_{s\in[0,1]}$ connecting $\gamma_0$ to $\gamma_1$ 
  along which $s\mapsto\Psi(\tau,u,v;\gamma_s)$ is uniformly convex of modulus $\frac3{2\tau}+\lambda$. 
\end{cond}
Then, firstly, the induction described above is well-defined, 
i.e., from each initial pair $(u_\tau^{\minus 1},u_\tau^{0})$ one obtains a unique discrete solution $(u_\tau^k)_{k\in\N}$.
Secondly, these discrete solutions converge to a curve $u_*$ of steepest descent.
Thirdly, the convergence can be made quantitative:
roughly, the distance of $u_\tau^k$ to the associated limit value $u_*(k\tau)$ is at most of order $\sqrt\tau$.

\subsection{Outline of the Paper}\label{subsec:Outline}
After recalling some basic notions from the theory of gradient flows in metric spaces in Section \ref{sct:ags},
we specify our hypotheses and discuss examples in Section \ref{sec:Assump},
Section \ref{sec:Pre} is devoted to the well-posedness of the BDF2 scheme, 
and to the derivation of various a priori estimates.
Section \ref{sec:Proof} is the heart of the paper,
with the statement and the proof of our main result on convergence, see Theorem \ref{thm:mainthm}.
In Section \ref{sec:Num}, we show numerical convergence of the methods on different examples.

\section{Notations for Gradient Flows in Abstract Metric Spaces}
\label{sct:ags}
We briefly recall a few basic notations from the theory of gradient flows in abstract metric spaces,
following \cite{ags}.
Here and below, $(\X,\bd)$ is a separable and complete metric space.

\begin{dfntn}[AC curves] 
  A curve $u: \left[0,\infty \right) \rightarrow \X$ 
  is said to be \emph{$L^2$-absolutely continuous}, written as $u\in \mathrm{AC}^2 \left( \left[0,\infty\right),\X\right)$, 
  if there exists a function $m \in L^2_{\mathrm{loc}} \left( \left[0,\infty\right) \right)$ 
  such that
  \begin{align*}
    \bd(u(t),u(s)) \leq \int_s^t m(r)  \dd r  \qquad \mbox{for all} \ \ 0\leq s \leq t. 
  \end{align*}
\end{dfntn}
See \cite[Definition 1.1.1]{ags} for further details.
It can be shown \cite[Theorem 1.1.2]{ags} that among all possible choices for $m$, 
there is a minimal one, 
called the \emph{metric derivative} $|u'|\in L^2_{\mathrm{loc}} \left( \left[0, \infty\right) \right)$, 
given by
\begin{align*}
  |u'|(t) := \lim_{s\rightarrow t} \frac{\bd(u(s),u(t))}{\left|s-t\right|}\qquad \mathrm{for \ a.e. \ } t. 
\end{align*}
The main definition is that of a \emph{gradient flow} in the energy landscape of a functional $\nrg:\X\to\Ru$
with respect to the metric $\bd$.
Here we adopt the strong but rather restrictive notion of the \emph{evolution variational inequality} (EVI).
This is particularly well adapted to dealing with gradient flows in the $L^2$-Wasserstein space $(\prb,\wass)$,
which will be one of our main examples.
\begin{dfntn}[EVI]
  Let a proper and $\bd$-lower-semicontinuous functional $\nrg:\X\to\Ru$ be given.
  We say that $\nrg$ generates a $\lambda$-contractive gradient flow on $(\X,\bd)$
  with some $\lambda\in\R$ if the following is true:

  For each initial condition $u_0\in\dom$, 
  there is a corresponding solution curve $u\in \mathrm{AC}^2 \left( \left[0,\infty\right),\X\right)$ with $u(0)=u_0$,
  that satisfies the evolution variational inequality associated to $\nrg$:
  \begin{align}
    \label{eq:EVI}
    \frac{1}{2} \bd^2(u(t) ,w)  - \frac12 \bd^2(u(s),w)  
    \leq   \int_s^{t} \left[\nrg(w)-\nrg(u(r)) -  \frac{\lambda}{2} \bd^2(u(r),w)\right]  \dd r  
  \end{align}  
  holds for arbitrary $0\le s\le t$ and for every $w\in\dom$.
\end{dfntn}
For a discussion of the role of the EVI in gradient flow theory, we refer to \cite[Chapter 4.0]{ags}.
\begin{rmrk}
  The inequalities \eqref{eq:EVI} are actually the integrated versions of the differential EVI,
  \begin{align}
    \label{eq:dEVI}
    \frac{d}{dt}\bd^2(u(t),w) +\frac\lambda2\bd^2(u(t),w) + \nrg(u(t)) \le\nrg(w),
  \end{align}
  which is supposed to hold for a.e. $t>0$.
  The integrated form \eqref{eq:EVI} is more convenient for our purposes,
  and it is easy to show that validity of \eqref{eq:EVI} for arbitrary $0\le s\le t$ implies
  validity of \eqref{eq:dEVI} for a.e. $t>0$.  
\end{rmrk}
The EVI is a more restrictive characterization of gradient flows than, e.g., the energy-dissipation-equality.
Most notably, validity of the EVI implies that the gradient flow is $\lambda$-contractive on $(\X,\bd)$,
so in particular, solutions are uniquely determined by their initial datum.
Moreover, if the metric space $(\X,\bd)$ is ``almost Euclidean'' 
--- for instance, $\X$ is a Hilbert space, 
or $\X$ is the space $\prb$ of probability measures endowed with the Wasserstein metric $\wass$ ---
then the EVI further implies that $\nrg$ is uniformly semi-convex, $\lambda$ being a modulus of convexity,
see \cite{Daneri}.
Thus, \eqref{eq:EVI} is not available for gradient flows of non-semi-convex functionals $\nrg$.

\section{Setup, Assumptions and Examples}\label{sec:Assump}
\subsection{Definition of the Method}\label{subsec:Disc}
Define the \emph{BDF2 penalization} $\Psi:(0,\tau_*)\times\X\times\X\times\X\to\Ru$ of $\nrg$ 
by
\begin{align*}
  \Psi(\tau,u,v; \cdot):  \X \rightarrow \Ru; \quad \Psi(\tau,u, v ;w ):= \frac1{ \tau}  \bd^2(v,w) - \frac1{4 \tau} \bd^2(u,w) +  \mathcal{E}(w). 
\end{align*}
The \emph{discrete solution} (for $\nrg$ on $(\X,\bd)$) corresponding to 
a time step size $\tau\in(0,\tau_*)$ and a pair of initial data $(u_\tau^{\minus 1},u_\tau^{0})\in\X\times\X$
is the sequence $(\uk)_\kinN$, 
which is inductively obtained via
\begin{align}
  u_\tau^{k}  \in \underset{w\in\X}{\mathrm{argmin}} \ \Psi(\tau,u_\tau^{k-2},u_\tau^{k-1};w)	\label{eq:BDF2scheme}
\end{align}
for $k\in\N$. 
Finally, the \emph{interpolated solution} $\bar{u}_\tau:\left[0,\infty\right) \rightarrow\X$ 
is obtained by piecewise constant interpolation of the values $\uk$ in time,
\begin{align}
  \bar{u}_\tau(0)=u_0, \quad \bar{u}_\tau(t)=u_\tau^k \quad \mathrm{for} \ t\in\left((k-1)\tau, k\tau\right] \quad \mathrm{and} \ \kinN. \label{eq:DiscSol}
\end{align}

\subsection{Main Assumptions}				
%
From now on, we assume that an energy functional $\nrg:\X \rightarrow \Ru$ with nonempty domain $\dom$ is given, 
which satisfies the following hypotheses.
\begin{enumerate}[({E}1)] 
\item  \textbf{Semi-continuity:} $\nrg$ is sequentially lower semi-continuous on $(\X,\bd)$:
  \begin{align*}
    u_n \dar u \qquad \Longrightarrow \qquad \mathcal{E} (u) \leq \liminf_\ntoinf \mathcal{E}(u_n).
  \end{align*}
\item \textbf{Coercivity:} There exist $\tau_* > 0$ and $u_* \in \X$, such that
  \begin{align*}
    c_* := \inf_{v\in \X} \frac{1}{2 \tau_*} \bd^2(u_*,v) + \mathcal{E}(v) > - \infty . 
  \end{align*}
\item \textbf{Semi-convexity:} There exists a constant $\lambda$ such that for every $u,v,\gamma_0,\gamma_1 \in \dom$ and every $\tau\in\left[0,\tau_*\right)$, 
  there exists a continuous curve $\gamma_{(\cdot)}:[0,1] \rightarrow \X$ joining the given end points $\gamma_0$ and $\gamma_1$,
  along which the penalized energy $\Psi$ satisfies
  \begin{align}
    \Psi(\tau,u,v;\gamma_s)  
    \leq (1-s) \Psi(\tau,u,v; \gamma_0) + s \Psi(\tau,u,v;\gamma_1) - \frac{1}{2} \left( \frac3{2\tau} + \lambda\right) s(1-s) \bd^2(\gamma_0,\gamma_1) . 
    \label{eq:Stronglambdaconvex}
  \end{align}
\end{enumerate}
Moreover, without loss of generality, we assume that
\begin{align}
  \label{eq:hypos}
  \lambda \le 0 \quad\text{and}\quad (-\lambda)\tau_* \le \frac12.
\end{align}
Note that for $0<\tau<\tau_*$, the last term on the right hand side of \eqref{eq:Stronglambdaconvex} is positive 
for $\gamma_0\neq \gamma_1$ and $0<s<1$,
implying that $s\mapsto\Psi(\tau,u,v;\gamma_s)$ is strictly convex.
\begin{rmrk}
  Assumptions (E1)\&(E2) are standard minimal hypotheses on the energy in the context of metric gradient flows.
  (E3) is a reformulation of Condition \ref{cond:bdf2} from the introduction.
  As mentioned there, it plays an analogous role for the BDF2 discretization 
  as Condition \ref{cond:ags} plays for the minimizing movement scheme.
\end{rmrk}

\subsection{Examples}
In this section we discuss three general situations in which the convexity assumption (E3) is satisfied,
namely that of uniformly semi-convex functionals $\nrg$ on a Hilbert space $H$, 
that of semi-convex $C^{1}$-functions $\nrg$ on Riemannian manifolds of non-negative cross-curvature,
and that of functionals $\nrg$ on the the $L^2$-Wasserstein space $(\prb,\wass)$
that are uniformly displacement semi-convex.

\subsubsection{Hilbert space}
Uniformly semi-convex functionals on Hilbert spaces provide a class of fairly easy examples
for the validity of assumption (E3), thanks to the linear structure of the space.
\begin{thrm}\label{thm:Hilberconvex}
  Assume that the metric space $(\X,\bd)$ is a Hilbert space $\X=H$, with the distance $\bd$ induced by the norm $\|\cdot\|$.
  Assume further that $\nrg$ is uniformly semi-convex with modulus $\lambda$.
  Then (E3) is satisfied, with $\gamma_s$ is the straight line between $\gamma_0,\gamma_1$ and with the same $\lambda$.
\end{thrm}
\begin{proof}
  Let $\gamma_0,\gamma_1\in\dom$ as well as $u,v\in\dom$ and $\tau>0$ be given.
  We verify \eqref{eq:Stronglambdaconvex} for the particular curve $\gamma_s:=(1-s)\gamma_0+s\gamma_1$.
  On the one hand, by the convexity hypothesis on $\nrg$, 
  we know that
 \begin{align}
    \nrg(\gamma_s) \leq (1-s) \nrg(\gamma_0) +(1-s)\nrg(\gamma_1) - \frac{\lambda}{2} s(1-s) \left\|\gamma_0-\gamma_1\right\|^2. \label{eq:Hlconvex1}
  \end{align}
  On the other hand, a direct calculation using the property of the scalar product yields
  \begin{align}
    \begin{split}
      &\| \gamma_s-v\|^2 - \frac14 \| \gamma_s - u \|^2 \\
      &= (1-s)\big( \| \gamma_0-v \|^2-\frac14 \| \gamma_0-u \|^2 \big) 
      + s\big( \| \gamma_1-v\|^2-\frac14 \| \gamma_1-u\|^2 \big) 
      - \frac{3}{4} s(1-s) \|\gamma_0-\gamma_1\|^2.      
    \end{split}
    \label{eq:Hlconvex2}
  \end{align}
  Adding $\frac1\tau$ times \eqref{eq:Hlconvex2} to \eqref{eq:Hlconvex1} yields \eqref{eq:Stronglambdaconvex}.
\end{proof}

\subsubsection{Riemannian manifolds}
\label{sct:diffgeo}
Another situation of interest is that of the gradient flow on a compact smooth Riemannian manifold $(\mf,\g)$,
which is induced by a semi-convex function $\nrg\in C^1(\mf)$.
Here, our very general approach is clearly not optimal:
in that finite dimensional setting, gradient flows can be characterized in a direct way instead of using the EVI \eqref{eq:EVI}.
Further, there are explicit and local variants of the BDF2 method (avoiding the global minimization of $\Psi$ in each time step), see e.g. \cite{HLgf},
which are more simple to implement, and whose convergence is expected under more easily varifiable hypotheses than (E3).
Still, for the sake of completeness, we shall detail a sufficient criterion for applicability of our results in that situation.

To indicate why the verification of (E3) indeed poses a (surprisingly hard) problem, 
observe that it is in general \emph{not} possible to use the geodesic $\tilde \gamma_{(\cdot)}$ for the curve connecting $\gamma_0$ to $\gamma_1$ in \eqref{eq:Stronglambdaconvex}.
Indeed, for $s\mapsto\Psi(\tau,u,v;\tilde\gamma_s)$ to be uniformly convex of modulus $\frac3{2\tau}+\lambda$,
independently of $u$ and $v$,
one would essentially need that both $s\mapsto\bd^2(u,\gamma_s)$ and $s\mapsto-\bd^2(v,\gamma_s)$ 
are uniformly convex of modulus $\bd^2(\gamma_0,\gamma_1)$.
By Toponogov's theorem, the first condition would imply that $\mf$ has non-negative sectional curvature, 
and the later would imply that $\mf$ has non-positive sectional curvature;
hence, $\mf$ would need to be flat.

A more appropriate class of connecting curves are \emph{segments}, 
which are defined with the help of the exponential map $\exp$ as follows.
Fix $u\in\mf$, and let $\gamma_0,\gamma_1\in\mf$ lie outside of $u$'s cut locus $\cut(u)$.
Then, there are unique $\xi_0,\xi_1$ in the injectivity domain $I(u)\subset T_x\mf$ of the exponential map $\exp_u:T_u\mf\to\mf$
such that $\exp_u(\xi_0)=\gamma_0$ and $\exp_u(\xi_1)=\gamma_1$.
Further, assume that the straight line from $\xi_0$ to $\xi_1$ lies in $I(u)$.
The segment $[\gamma_0,\gamma_1;u]_{(\cdot)}:[0,1]\to\mf$ with base $u$ connecting $\gamma_0$ to $\gamma_1$
is then defined by $[\gamma_0,\gamma_1;u]_s=\exp_u((1-s)\xi_0+s\xi_1)$.

Kim and McCann \cite[Corollary 2.11]{kim2012towards} have established a sufficient criterion for the convexity of
\begin{align}
  \label{eq:helpf}
  [0,1]\ni s\mapsto\bd^2\big(u,[\gamma_0,\gamma_1;u](s)\big) -\bd^2\big(v,[\gamma_0,\gamma_1;u]_s\big),
\end{align}
independently of $v\in\mf$.
Their hypotheses are as follows.
\begin{enumerate}
\item[(KM0)] The squared metric $\bd^2(\cdot,\cdot)$, induced on $\mf$ via $\g$, is $C^4$-regular outside of the cut locus.
\item[(KM1)] For each $u\in\mf$, its injectivity domain $I(u)$ is convex, 
  so segments $[\gamma_0,\gamma_1;u]$ can be defined for arbitrary $\gamma_0,\gamma_1\notin\cut(u)$.
\item[(KM2)] For each segment $[\gamma_0,\gamma_1;u]$, there is a dense subset $V\subset\mf$, 
  such that there is no $v\in V$ and no $s\in[0,1]$ with $v\in\cut([\gamma_0,\gamma_1;u]_s)$. 
\item[(KM3)] $(\mf,\g)$ has non-negative cross curvature; the definition is given in Appendix \ref{apx:cross}.
\end{enumerate}
Apart from (KM0), each of these conditions is rather demanding.
A class of examples satisfying (KM0)--(KM3) are the round spheres $\Sd$.
For these, (KM0)--(KM2) are easily verified since $\cut(u)=\{-u\}$ only contains the antipodal point,
and $I(u)$ is the open $d$-dimensional ball of radius $\pi$, for each $u\in\Sd$.
In contrast, the proof of (KM3) has been a challenge even for spheres, that has been mastered in \cite[Theorem 6.2]{kim2012towards}.
It seems that --- apart from products and quotients of spheres --- no further explicit examples 
satisfying (KM0)--(KM3) are currently known.
\begin{thrm}
  \label{thm:diffgeo}
  Assume that $(\X,\bd)$ is a compact Riemannian manifold $(\mf,\g)$ that satisfies (KM0)--(KM3) above.
  Assume further that $\nrg\in C^{1}(\mf)$ is semi-convex.
  Then (E3) is satisfied with some $\lambda\in\R$, with $\gamma_{(\cdot)}:=[\gamma_0,\gamma_1;u]_{(\cdot)}$.
\end{thrm}
\begin{proof}
  For given $u,v\in\mf$ and $\gamma_0,\gamma_1\in\mf\setminus\cut(u)$, let $\gamma_{(\cdot)}:=[\gamma_0,\gamma_1;u]_{(\cdot)}$;
  the result for general $\gamma_0,\gamma_1\in\mf$ follows by continuity a forteriori.
  Further, we shall assume that $\nrg\in C^2(\mf)$ during the computations.
  Since $\nrg$ is semi-convex, and $\mf$ is compact, there is a global modulus $\lambda'\le0$ of convexity, 
  i.e., $\hess\nrg(u)\ge\lambda'$ as a quadratic form on each $T_u\mf$.
  The final estimate \eqref{eq:manifold} depends only on $\lambda'$, 
  so (E3) follows for general semi-convex $\nrg\in C^1(\mf)$ by approximation.

  We split \[\Phi(\tau,u,v;\gamma_s)=h_1(s)+h_2(s)+h_3(s),\]
  with $h_1,h_2,h_3:[0,1]\to\R$ given by
  \begin{align*}
    h_1(s) = \frac3{4\tau}\bd^2(u,\gamma_s), \quad
    h_2(s) = \frac1{4\tau}\big(\bd^2(u,\gamma_s)-\bd^2(v,\gamma_s)\big), \quad
    h_3(s) = \nrg(\gamma_s).
  \end{align*}
  First, by definition of the segment $\gamma_{(\cdot)}$ via the exponential map,
  $s\mapsto\bd^2(u,\gamma_s)$ is twice differentiable with
  \begin{align*}
    \frac3{4\tau}\frac{d^2}{ds^2}\bd^2(u,\gamma_s) \equiv \frac3{2\tau}\|\xi_1-\xi_0\|_u^2,
  \end{align*}
  where $\|\xi\|_u^2=\g_u(\xi,\xi)$. 
  Second, by the hypotheses (KM0)--(KM3), the result from \cite[Corollary 2.11]{kim2012towards} applies, so $h_2$ is convex.
  Finally, concerning $h_3$:  
  in the normal coordinates induced by $\exp_u:I(u)\to\mf$, the segment $\gamma_{(\cdot)}$ is the straight line connecting $\xi_0$ to $\xi_1$,
  hence (recalling the definition of the Hessian, and that $\exp_u$ is a 1-Lipschitz map):
  \begin{align*}
    h_3''(s) = \frac{d^2}{ds^2}\nrg(\gamma_s) 
    &= \hess\nrg(\gamma_s)[\dot\gamma_s] + d\nrg(\gamma_s)\big[\nabla_{\dot\gamma_s}\dot\gamma_s\big]\\
    &\ge \lambda'\|\dot\gamma_s\|_{\gamma_s}^2 -
      \|\nrg\|_{C^1}\left\|\sum_{i,j,k}\Gamma_{i,j}^k(\xi_1^i-\xi_0^i)(\xi_1^j-\xi_0^j)\frac{\partial}{\partial x_k}\right\|_{\gamma_s} 
    \ge \big(\lambda'-K\|\nrg\|_{C^1}\big)\|\xi_1-\xi_0\|_u^2.
  \end{align*}
  Here $K$ is a bound on the Christoffel symbols $\Gamma_{ij}^k$ on the smooth and compact manifold $(\mf,\g)$,
  and for the estimate $\|\dot\gamma_s\|\le\|\xi_1-\xi_0\|_u$, 
  we have used that (KM3) implies that $(\mf,\g)$ is of non-negative sectional curvature.

  In summary, we have shown that $s\mapsto\Phi(\tau,u,v;\gamma_s)$ is uniformly convex of modulus
  \begin{align}
    \label{eq:manifold}
    \left(\frac3{4\tau}+\lambda\right)\|\xi_1-\xi_0\|_u^2 \quad \text{with} \quad \lambda:=\lambda'-K\|\nrg\|_{C^1}.
  \end{align}
  Recalling that (KM3) implies non-negative sectional curvature on $(\mf,\g)$,
  we conclude that $\bd^2(\gamma_0,\gamma_1)\le\|\xi_1-\xi_0\|_u^2$,
  so the claim (E3) follows.
\end{proof}

\subsubsection{Wasserstein space}
In our last example, we consider the $L^2$-Wasserstein space $\left(\prb,\wass\right)$ 
of the probability measures of finite second moment over a convex, possibly unbounded domain $\Omega\subset\Rd$. 
And we assume that $\nrg$ is uniformly displacement semi-convex; the definition is recalled below.
We remark that the class of gradient flows generated in this setting encompasses nonlinear drift-diffusion-aggregation equations
of the form
\begin{align*}
  \partial_tu = \Delta(u^m) + \nabla\cdot(u\nabla V) + \nabla\cdot(u\ast\nabla W\,u),
\end{align*}
under the restrictions that $m\ge(d-1)/d$, and that $V,W\in C^2(\Omega)$ are uniformly semi-convex.

$\left(\prb,\wass\right)$ is a complete geodesic space, which has non-negative curvature in the sense of Alexandrov.
Similarly as in the case of (non-negatively cross-curved) Riemannian manifolds discussed above, 
one cannot expect that hypothesis (E3) is satisfied for the geodesic $\tilde\gamma_{(\cdot)}$ 
connecting the two given measures $\gamma_0,\gamma_1\in\prb$.
Indeed, $s\mapsto\wass^2(\tilde\gamma_s,u)$ is typically \emph{not} uniformly convex of modulus $\wass^2(\gamma_0,\gamma_1)$,
see \cite[Example 7.3.3]{ags}.
Again, segments with prescribed base point are more appropriate.

We need to recall some basic notations from the theory of optimal mass transport.
$\pgen(\Omega_j\times\Omega_k)$ is the space of probability measures with finite second moment on the cross product $\Omega\times\Omega$,
and the indices $j$ and $k$ indicate that we use coordinates $x_j\in\Omega$ and $x_k\in\Omega$ on the components, 
i.e., we write $x=(x_j,x_k)\in\Omega\times\Omega$.
We introduce the canonical projections $\pi_j:(x_j,x_k)\mapsto x_j$,
and for $t\in[0,1]$, we write $(1-t)\pi_0+t\pi_1$ for brevity.
We write $\pi_j\#\bmu$ for the $j$-marginal of $\bmu\in\pgen(\Omega_j\times\Omega_k)$,
and analogously, for $\bmu\in\pgen(\Omega_0\times\Omega_1)$ and $t\in[0,1]$,
the interpolating measure $\pi_t\#\bmu$ is characterized by
\begin{align*}
  \int \varphi(y)\,d\pi_t\#\bmu(y) = \int \varphi\big((1-t)x_0+tx_1\big)\,d\bmu(x),
  \quad
  \text{for all $\varphi\in C^0_b(\Omega)$}.
\end{align*}
A transport plan from $\mu_0\in\pgen(\Omega_0)$ to $\mu_1\in\pgen(\Omega_1)$ 
is any $\bmu\in\pgen(\Omega_0\times\Omega_1)$ satisfying 
the marginal constraints $\pi_0\#\bmu=\mu_0$ and $\pi_1\#\bmu=\mu_1$.
Such a plan $\bmu$ is called optimal if it is a minimzer in the Kantorovich problem 
\begin{align}
  \label{eq:kantorovich}
  \bmu \mapsto \int |x_0-x_1|^2\,d\bmu(x).
\end{align}
For any given $\mu_0,\mu_1\in\pgen(\Omega)$, there exists at least one optimal plan.
The corresponding minimal value of the integral in \eqref{eq:kantorovich} defines
the squared Wasserstein distance $\wass^2(\mu_0,\mu_1)$.

We are going to use the following two facts,
which are essentially \cite[Lemma 5.3.2]{ags} and \cite[Proposition 7.3.1]{ags}:
\begin{enumerate}
\item \emph{Glueing lemma:} Given $\alpha\in\pgen(\Omega_0\times\Omega_2)$ and $\beta\in\pgen(\Omega_1\times\Omega_2)$
  with $\pi_2\#\alpha=\pi_2\#\beta$,
  there exists a $\mu\in\pgen(\Omega_0\times\Omega_1\times\Omega_2)$ such that
  $(\pi_0,\pi_2)\#\mu=\alpha$ and $(\pi_1,\pi_2)\#\mu=\beta$.
\item \emph{Curve lemma:} Given $\alpha\in\pgen(\Omega_0\times\Omega_1)$, $\beta\in\pgen(\Omega_3)$ and $t\in[0,1]$,
  there exists a $\mu\in\pgen(\Omega_0\times\Omega_1\times\Omega_3)$ such that
  $(\pi_0,\pi_1)\#\mu=\alpha$, and $(\pi_t,\pi_3)\#\mu$ is an optimal transport plan from $\pi_t\#\alpha$ to $\beta$.
\end{enumerate}
Segments --- which are referred to as generalized geodesics in \cite{ags} --- are defined as follows.
Let $\bmu_{02}\in\pgen(\Omega_0\times\Omega_2)$ and $\bmu_{12}\in\pgen(\Omega_1\times\Omega_2)$ be optimal plans 
for the transport of $\gamma_0$ and $\gamma_1$, respectively, to $u$.
By the glueing lemma, there exists a $\bmu_{012}$ 
such that $(\pi_0,\pi_2)\#\bmu_{012}=\bmu_{01}$ and $(\pi_1,\pi_2)\#\bmu_{012}=\bmu_{12}$.
Then $[\gamma_0,\gamma_1;u]_s:=\pi_s\#\bmu_{012}$.
Finally, we recall that $\nrg$ being uniformly displacement semi-convex of modulus $\lambda$ means that
\begin{align*}
  \nrg\big([\gamma_0,\gamma_1;u]_s\big)
  \le(1-s)\nrg(\gamma_0)+s\nrg(\gamma_1) - \frac\lambda2\int|x_0-x_1|^2\,d\bmu_{012}(x).
\end{align*}
\begin{thrm}
  Assume that the metric space $(\X,\bd)$ is the $L^2$-Wasserstein space $(\prb,\wass)$, with $\Omega\subset\Rd$.
  Assume further that $\nrg$ is uniformly displacement semi-convex of modulus $\lambda$.
  Then (E3) is satisfied, with the same $\lambda$, for $\gamma_{(\cdot)}=[\gamma_0,\gamma_1;u]_{(\cdot)}$.
\end{thrm}
\begin{proof}
  Let $u,v,\gamma_0,\gamma_1\in\prb$ be given, and let $\bmu_{012}$ be as above.
  We are going to prove the inequality \eqref{eq:Stronglambdaconvex} directly for a fixed value $s\in(0,1)$.

  Since $(\pi_s,\pi_2)\#\bmu_{012}$ is \emph{some} transport from $\gamma_s$ to $u$,
  and $(\pi_0,\pi_2)\#\bmu_{012}$, $(\pi_1,\pi_2)\#\bmu_{012}$ are both optimal,
  \begin{align*}
    \wass^2(\gamma_s,u) 
    &\le \int |(1-s)x_0+sx_1-x_2|^2\,d\bmu_{012}(x) \\
    &= \int \big[(1-s)|x_0-x_2|^2+s|x_1-x_2|^2-s(1-s)|x_0-x_1|^2\big]\,d\bmu_{012}(x) \\
    &=(1-s)\wass^2(\gamma_0,u)+s\wass^2(\gamma_1,u)-s(1-s)\int|x_0-x_1|^2\,d\bmu_{012}(x).
  \end{align*}
  By the curve extension lemma, there exists a $\bmu_{013}\in\pgen(\Omega_0\times\Omega_1\times\Omega_3)$,
  such that $(\pi_0,\pi_1)\#\bmu_{013}=(\pi_0,\pi_1)\#\bmu_{012}$, 
  and $(\pi_s,\pi_3)\#\bmu_{013}$ is an optimal plan from $\gamma_s$ to $v$.
  It follows that $(\pi_0,\pi_3)\#\bmu_{013}$ and $(\pi_1,\pi_3)\#\bmu_{013}$ 
  are some transport plans from $\gamma_0$ and $\gamma_1$, respectively, to $v$,
  and so
  \begin{align*}
    \wass^2(\gamma_s,v) 
    &= \int |(1-s)x_0+sx_1-x_3|^2\,d\bmu_{013}(x) \\
    &= \int \big[(1-s)|x_0-x_3|^2+s|x_1-x_3|^2-s(1-s)|x_0-x_1|^2\big]\,d\bmu_{013}(x) \\
    &\ge(1-s)\wass^2(\gamma_0,v)+s\wass^2(\gamma_1,v)-s(1-s)\int|x_0-x_1|^2\,d\bmu_{012}(x).
  \end{align*}
  In combination with the definition of $\lambda$-uniform displacement convexity of $\nrg$,
  we arrive at 
  \begin{align*}
    \Psi(\tau,u,v;\gamma_s)  
    \leq (1-s) \Psi(\tau,u,v; \gamma_0) + s \Psi(\tau,u,v;\gamma_1) - \frac{1}{2} \left( \frac3{2\tau} + \lambda\right) s(1-s) \int |x_1-x_0|^2\,d\bmu_{012}(x).
  \end{align*}
  Clearly, the integral above is larger or equal to $\wass^2(\gamma_0,\gamma_1)$, 
  hence \eqref{eq:Stronglambdaconvex} for any $\tau>0$ so small that $\frac3{2\tau}+\lambda>0$.
\end{proof}

\section{Well-posedness of the Scheme and Classical Estimates}
\label{sec:Pre}
In this section, we study basic properties of the BDF2 scheme.
First, we prove well-posedness in the sense that 
for all sufficiently small $\tau>0$, and arbitrary data $u,v\in\dom$,
the functional $\Psi(\tau,u,v;\cdot)$ possesses a unique minimizer in $\dom$.
Consequently, for an arbitrary pair $(u_\tau^{\minus 1},u_\tau^{0})$ of initial conditions, 
one obtains inductively a unique global discrete solution $(u^k_\tau)_\kinN$ 
by solving the corresponding sequence of minimization problems in \eqref{eq:BDF2scheme}.
Subsequently, we derive some fundamental estimates that are needed for the convergence proof later.

Recall that Assumptions (E1)--(E3) are supposed to hold, with \eqref{eq:hypos}.
\begin{thrm}[Existence of a minimizer] \label{ExistenceBDF2}
  For all $\tau \in \left(0 , \tau_* \right)$ and for all $ u,v \in \X$, 
  there exists a unique minimizer $w_* \in \dom$ of $w \mapsto \Psi(\tau,u,v; w)$.
\end{thrm}
\begin{proof}
  Let $\tau \in \left(0 , \tau_* \right)$ and $ u,v \in \X$ be fixed.
  For brevity, we write $\psi(w):=\Psi(\tau,u,v;w)$.

  First, we show that $\psi$ is bounded from below.
  By the triangle inequality and the binomial formula, we have that
  \begin{align*}
    \bd^2(u,w) \leq 2 \bd^2(u,v) + 2 \bd^2(v,w), 
    \quad
    \bd^2(u_*,w)& \leq \frac{2\tau_*}{\tau_*+ \tau} \bd^2(w,v) + \frac{2\tau_*}{\tau_* -\tau} \bd^2(v,u_*).
  \end{align*}
  Substituting these estimates into the definition of $\psi(w)=\Psi(\tau,u,v;w)$
  and using Assumption (E2), we obtain for each $w\in\dom$:
  \begin{align*}
    \psi(w)
    \ge & \ \ \ \, \frac1{\tau_* + \tau}\bd^2(v,w) + \frac1{2\tau} \bd^2(v,w) - \frac1{4\tau} \bd^2(u,w) + \nrg(w) \\
    \geq &  - \frac{1}{\tau_* -\tau} \bd^2(v,u_*) + \frac1{2\tau_*}\bd^2(u_*,w) -\frac1{2\tau} \bd^2(v,u)+ \nrg(w) \\
    \geq &  - \frac{1}{\tau_* -\tau} \bd^2(v,u_*) -\frac1{2\tau} \bd^2(v,u) + c_*.
  \end{align*}
  The last expression, which only depends on the given quantities $u$ and $v$,
  constitutes the sought for lower bound on $\psi$.
  Consequently,
  \begin{align*}
    \underline\psi := \inf_{w\in\dom}\psi(w) > -\infty.
  \end{align*}
  Now, choose a minimizing sequence $\left( w_n \right)_\ninN$ in $\dom$, 
  i.e., 
  \begin{align}
    \label{eq:precauchy}
    \psi( w_n)\searrow\underline\psi.    
  \end{align}
  We are going to prove that this is a Cauchy sequence.
  Towards that goal, we invoke Assumption (E3):
  specifically, for given indices $m$ and $n$, 
  we choose $\gamma_0= w_m$, $\gamma_1= w_n$, 
  and we define $ w_{m,n}:=\gamma_{\frac12}$, the midpoint of the respective curve joining $ w_m$ to $ w_n$.
  Then, by \eqref{eq:Stronglambdaconvex},
  \begin{align*}
    \psi( w_{m,n}) 
    \leq \frac12 \psi( w_m) + \frac12 \psi( w_n) - \frac{1}{8} \left( \frac3{2\tau} + \lambda\right) \bd^2(w_m, w_n).
  \end{align*}
  Since $\tau < \tau_*$ by hypothesis, and $3+2\lambda\tau_*\ge2$ thanks to \eqref{eq:hypos}, 
  this yields an estimate on the distance from $w_m$ to $w_n$:
  \begin{align*}
    \bd^2( w_m, w_n) 
    \le \frac{8\tau}{3+2\tau\lambda} \left( \psi(w_m)+ \psi( w_n) - 2 \psi(w_{m,n}) \right)
    \le \frac{8\tau}{3+2\tau\lambda}  \left( \psi(w_m)+ \psi( w_n) - 2  \underline\psi \right). 
  \end{align*}
  In view of \eqref{eq:precauchy}, this verifies the Cauchy property of $\left( w_n \right)_\ninN$.
  Consequently, and by completeness of $(\X,\bd)$, that sequence converges to a limit $w_*\in\X$.

  According to Assumption (E1), $\nrg$ is lower $\bd$-semi-continuous.
  Since the distance to a given point is clearly a continuous function, also $\psi$ is lower $\bd$-semi-continuous.
  By the usual argument
  \begin{align*}
    \underline\psi \le \psi(w_*) \le \liminf_{n\to\infty}\psi( w_n) = \underline\psi,
  \end{align*}
  we conclude that $\psi$ attains its infimum $\underline\psi$ at $w_*$, 
  i.e., $w_*$ is a minimizer.

  Uniqueness of the minimizer follows by Assumption (E3) as well:
  by the remarks following \eqref{eq:hypos}, $\psi$ is \emph{strictly} convex 
  along some curve that connects two potentially different minimizers.
  But that would mean that $\psi$ attains a value lower than that at the minimizers, a contradiction.
\end{proof}
In the following, we assume that discrete initial data $(u_\tau^{\minus 1},u_\tau^{0})$ are given for each $\tau\in(0,\tau_*)$,
and we consider the --- according to Theorem \ref{ExistenceBDF2} above --- well-defined family of discrete solutions $(u_\tau^k)_{k\in\N}$.
We recall that one of the key features of the implicit Euler method is that the energy values $\nrg(u_\tau^k)$ 
are  monotonically decreasing with $k$.
This is not quite the case for the BDF2 scheme at hand, but we can prove a slightly weaker property.
\begin{lmm}[Almost energy diminishing]\label{lem:EnergyDim}
  Each discrete solution $(u_\tau^k)_{k\in\N}$ satisfies
  \begin{align}
    \nrg(\uk) + \frac1{2\tau} \bd^2(\ukm,\uk) \leq \nrg(\ukm) + \frac{1}{4\tau} \bd^2(\ukzm,\ukm)
    \label{eq:EnergyDim}
  \end{align}
  at each step $k=1,2,\ldots$.
\end{lmm}
\begin{proof}
  Since $\uk$ is a minimizer of $w \mapsto \Psi(\tau,\ukzm,\ukm; w)$, it satisfies 
  \begin{align*}
    \Psi(\tau,u_\tau^{k-2},u_\tau^{k-1};u_\tau^{k}) \leq \Psi(\tau,u_\tau^{k-2},u_\tau^{k-1};w) 
  \end{align*}	
  for all $w\in\X$.
  For the choice $w= u_\tau^{k-1}$, we obtain
  \begin{align}
    \label{eq:predim}
    \frac{1}{\tau} \bd^2(\ukm,\uk) - \frac{1}{4\tau} \bd^2(\ukzm,\uk) + \nrg(\uk)  
    \leq - \frac{1}{4\tau} \bd^2(\ukzm,\ukm) + \nrg(\ukm).
  \end{align}
  By the triangle inequality and the binomial formula,
  \begin{align}
    \bd^2(\ukzm,\uk) \le 2\bd^2(\ukzm,\ukm) + 2\bd^2(\ukm,\uk). \label{eq:binomf}
  \end{align}
  Substitute this in the left-hand side of \eqref{eq:predim}.
  This yields \eqref{eq:EnergyDim}
\end{proof}
Next, we derive the classical estimates on energy and distance.
These estimates require some further assumptions on the discrete initial data $(u_\tau^{\minus 1},u_\tau^{0})$:
there are constants $K_0$, $K_1$ and $K_2$, such that, for all $\tau\in(0,\tau_*)$,
  \begin{itemize}
  \item[(I0)] $\bd(u_\tau^0,u^0)\le K_0\tau$, and
  \item[(I1)] $\nrg(u_\tau^0)\le K_1$ and $\nrg(u_\tau^{-1})\le K_1$, and
  \item[(I2)] $\bd(u_\tau^{\minus 1},u_\tau^{0})\le K_2\tau$.
  \end{itemize}
\begin{thrm}[Classical estimates] 
  \label{BoundsMMS}
  Fix a time horizon $T>0$.
  Under the aforementioned assumptions on the discrete initial data $(u_\tau^{\minus 1},u_\tau^{0})$,
  there is a constant $C$, depending only on $K_0$ to $K_2$ and $T$, 
  such that the corresponding discrete solutions $(u_\tau^k)_{k\in\N}$ satisfy
  \begin{align}
    \sum_{k=0}^{N} \frac{1}{2\tau} \bd^2(\ukm,\uk) & \leq  C , \label{boundsumdk} \\
    |\mathcal{E}(u_\tau^{N}) | & \leq C, \label{BoundEk} \\
    \bd^2(u_*,u_\tau^{N}) & \leq C \label{bounddn},
  \end{align}
  for all $N\in\N$ with $N\tau\le T$.
\end{thrm}
\begin{proof}
  The main estimate is easy to obtain: 
  sum up the inequalities \eqref{eq:EnergyDim} for $k=1$ to $k=N$.
  After cancellation of corresponding terms on both sides, we remain with
  \begin{align}
    \label{eq:bound1}
    \nrg(u_\tau^N)+\frac1{4\tau}\sum_{k=1}^N\bd^2(u_\tau^{k-1},u_\tau^{k})
    \le \nrg(u_\tau^0)+\frac1{4\tau}\bd^2(u_\tau^{\minus 1},u_\tau^{0})
    \le K_1+\frac14K_2^2,
  \end{align}
  where we have used the hypotheses (I1)\&(I2) for the last inequality.
  Clearly, if $\nrg$ would be bounded below, 
  then \eqref{boundsumdk}--\eqref{bounddn} would follow immediately.

  Since we only assume the weaker lower bound (E2), more work is required.
  First, we show that
  \begin{align}
    \label{eq:bound2}
    \bd^2(u_*,u_\tau^{k}) -\bd^2(u_*,u_\tau^{k-1}) \le 2\bd(\ukm,\uk)\bd(u_*,\uk).
  \end{align}
  We only need to consider the case that $\bd(u_*,u_\tau^{k})\ge\bd(u_*,u_\tau^{k-1})$,
  since otherwise the inequality is trivially true.
  But then an application of the triangle inequality yields:
  \begin{align*}
    \bd^2(u_*,u_\tau^{k}) -\bd^2(u_*,u_\tau^{k-1}) 
    &= \big(\bd(u_*,u_\tau^{k}) +\bd(u_*,u_\tau^{k-1})\big)\big(\bd(u_*,u_\tau^{k}) -\bd(u_*,u_\tau^{k-1})\big)\\
    &\le \big(\bd(u_*,u_\tau^{k})+\bd(u_*,u_\tau^{k})\big)\big(\bd(u_*,u_\tau^{k-1}) + \bd(u_\tau^{k-1},u_\tau^{k}) -\bd(u_*,u_\tau^{k-1})\big) \\
    &= 2\bd(\ukm,\uk)\bd(u_*,\uk),
  \end{align*}
  which is \eqref{eq:bound2}.
  We use \eqref{eq:bound2} and the binomial formula to estimate
  \begin{align*}
    \frac{1}{2}	\bd^2(u_*,u_\tau^{N}) - \frac{1}{2} \bd^2(u_*,u_\tau^{0}) 
    &=  \frac12 \sum_{k=1}^{N} \big[\bd^2(u_*,u_\tau^{k}) -\bd^2(u_*,u_\tau^{k-1})\big] \\ 
    &\leq \sum_{k=1}^{N}\bd(\ukm,\uk)\bd(u_*,\uk)
    \leq  \sum_{k=1}^{N} \frac{\tau_*}{8 \tau } \bd^2(\ukm,\uk) +  \sum_{k=1}^{N} \frac{2 \tau}{\tau_*} \bd^2(u_*, \uk) .
  \end{align*}
  At this point, we substitute estimate \eqref{eq:bound2} and use Assumption (E2) to obtain
  \begin{align*}
    \frac{1}{2}	\bd^2(u_*,u_\tau^{N}) - \frac{1}{2} \bd^2(u_*,u_\tau^{0}) 
    &  \leq \frac{ \tau_*}{2} \left( \nrg(u_\tau^0) - \nrg(u_\tau^{N}) +\frac1{4\tau}\bd^2(u_\tau^{\minus 1},u_\tau^{0})\right) 
      +  \frac{2 \tau}{\tau_*} \sum_{k=1}^{N} \bd^2(u_*, \uk) \\
    &\leq \frac{ \tau_*}{2} \left( \nrg(u_\tau^0) - c_* + \frac{1}{2\tau_*}\bd^2(u_*, u_\tau^{N}) +\frac1{4\tau}\bd^2(u_\tau^{\minus 1},u_\tau^{0})\right) 
      +  \frac{2 \tau}{\tau_*}\sum_{k=1}^{N} \bd^2(u_*, \uk).
  \end{align*}
  We rearrange terms and use (I0)--(I2) to arrive at the following time-discrete Gronwall inequality:
  \begin{align*}
    \bd^2(u_*,u_\tau^{N}) 
    \leq 2K_0^2 +  2 \tau_* \left( K_1 - c_* \right) + \frac{\tau_*}2K_2^2 + \frac{8 \tau  }{\tau_*}\sum_{k=1}^{N} \bd^2(u_*, \uk).
  \end{align*}
  One verifies by induction on $N$ that 
  \begin{align*}
    \bd^2(u_*,u_\tau^{N}) 
    \le \left[2K_0^2 +  2 \tau_* \left( K_1 - c_* \right) + \frac{\tau_*}2K_2^2\right]
    \left(1+\frac{8\tau}{\tau_*}\right)^N
    \le \hat C\exp\left(\frac{8N\tau}{\tau_*}\right)
    \le \hat C\exp\left(\frac{8T}{\tau_*}\right).
  \end{align*}
  In other words, we have proven \eqref{bounddn}.
  
  From here, we conclude \eqref{BoundEk}:
  the bound on $\nrg(u_\tau^N)$ from above follows immediately from \eqref{eq:bound1},
  for the bound from below, we combine \eqref{bounddn} with Assumption (E2).
  Having \eqref{BoundEk} at hand, the bound \eqref{boundsumdk} follows again from \eqref{eq:bound1}.
\end{proof}
\begin{rmrk}
  We emphasize that the boundedness of the discrete velocity \eqref{boundsumdk} 
  reflects a crucial stability property of the BDF2 scheme.
  Essentially, it prevents the discrete solutions to oscillate rapidly or diverge to infinity, at least in the eyes of the metric $\bd$, 
  which is typically rather weak, but still the key element for all further convergence analysis.
  The fact that the scheme allows such a stability estimate is by no means a triviality;
  indeed, it is a consequence of the intrinsic \emph{A-stability} of the BDF2 method for ordinary differential equations, 
  see e.g. \cite{dahlquist1956convergence,dahlquist1963special,BD,Hairer2013}.
  Since the implicit Euler (BDF1) and the BDF2 methods are the only two A-stable Backward-Differentiation-Formulas, 
  it cannot be expected that an estimate of the form \eqref{boundsumdk} can be proven for any higher order BDF$k$ method.
\end{rmrk}
As a final preparation for the convergence proof,
we derive a time-discrete version of the differential EVI \eqref{eq:dEVI}.
That estimate does not require any further assumptions on the discrete initial data.
\begin{lmm}[Discrete EVI]\label{lem:VarIneq}
  The discrete solution $\left( \uk \right)_\kinN$ satisfies 
  \begin{align}
    \label{eq:VarIneq}
    \begin{split}
      &\left(\frac3{4\tau}+\frac{\lambda}{2}\right) \bd^2(\uk,w)
      - \frac1{\tau} \bd^2(\ukm,w) + \frac1{4\tau} \bd^2(\ukzm,w) \\
     &\leq  \nrg(w)-\nrg(\uk) - \frac1{\tau} \bd^2(\ukm,\uk) + \frac1{4\tau}\bd^2(\ukzm, \uk). 
    \end{split}
  \end{align}			
  for all $\kinN$, and for all $w\in\dom$.
\end{lmm}	
\begin{proof} 
  This follows from Assumption (E3).
  Choose $\gamma_0=\uk$ and $\gamma_1=w$, and let $(\gamma_s)_{s\in[0,1]}$ be the corresponding connecting curve 
  such that \eqref{eq:Stronglambdaconvex} holds.
  Combine \eqref{eq:Stronglambdaconvex} with the fact that $\uk$ minimizes $\Psi(\tau,\ukzm,\ukm,\cdot)$ to obtain,
  for each $s\in(0,1)$,
  \begin{align*}
    0  & \le \Psi(\tau,\ukzm,\ukm;\gamma_s) - \Psi(\tau,\ukzm,\ukm;\uk)  \\ 
       & \le s\Psi(\tau,\ukzm,\ukm;w) - s\Psi(\tau,\ukzm,\ukm;\uk) - \frac12 \left( \frac3{2\tau}+ \lambda\right) s(1-s) \bd^2(\uk,w) .
  \end{align*}
  Divide by $s\in(0,1)$ and pass to the limit $s\searrow 0$. 
  This yields
  \begin{align*}
    0 \leq \Psi(\tau,\ukzm,\ukm;w) - \Psi(\tau,\ukzm,\ukm;\uk) - \frac12 \left( \frac3{2\tau} + \lambda\right) \bd^2(\uk,w) ,
  \end{align*}
  which, by definition of $\Psi$, is the desired inequality \eqref{eq:VarIneq}.
\end{proof}


\section{Convergence}\label{sec:Proof}
\subsection{Statement of the Main Result}
Once again, we recall that $(\X,\bd)$ is a separable and complete metric space,
on which a functional $\nrg:\X\to\Ru$ is given, that satisfies Assumptions (E1)-(E3), with \eqref{eq:hypos}.
Our main result is the following.
\begin{thrm}[Convergence result] \label{thm:mainthm}
  Consider a vanishing sequence $(\tau_n)_\ninN$ of step sizes $\tau_n\in (0,\tau_*)$ that is strictly decreasing,
  and which is such that the quotients $\tau_{n}/\tau_{n+1}$ are natural numbers.
  Let further initial data $(u_{\tau_n}^{\minus 1},u_{\tau_n}^{0})$ be given 
  that satisfy the hypotheses (I0)--(I2) with appropriate $n$-independent constants $K_0$ to $K_2$,
  and in addition, there is a $K_3$ such that
  \begin{align}
    \label{eq:preroot}
    \bd(u_{\tau_n}^0,u^0)\le K_3 \tau_n.
  \end{align}
  For each $n$, the associated discrete solution $(u_{\tau_n}^k)_{k\in\N}$ is well-defined.

  Then the sequence of piecewise constant interpolations $(\bar{u}_{\tau_n})_{n\in\N}$ 
  converges locally uniformly with respect to time 
  to a solution $u_*\in \mathrm{AC}^2 \left( \left[0,\infty\right),\X\right)$ of the gradient flow for $\nrg$,
  i.e., the limit $u_*$ satisfies \eqref{eq:EVI}. 
  More precisely, for every time horizon $T>0$, 
  there is a constant $C$ that can be expressed in terms of $K_0$ to $K_3$ and $T$ alone,
  such that
  \begin{align}
    \label{eq:root}
    \bd\big( \bar{u}_{\tau_n}(t),u_*(t)\big)\leq C \sqrt{\tau_n}
  \end{align}
  for all $t\in[0,T]$.
\end{thrm}
\begin{rmrk}
  The hypothesis that consecutive $\tau_n$'s have an integer quotient has been made 
  in order not to make the already quite technical proof even more technical.
  Under that hypothesis, the time grid associated to some $\tau_n$ is always a refinement of the grid for $\tau_m$ if $n>m$. 
  That simplifies our calculations considerably.
\end{rmrk}

%
\begin{rmrk}
  \label{rmk:counter}
  We give a simple example showing that under the given assumptions,
  in general one cannot expect second order convergence of the BDF2 method, 
  i.e., $\tau_n^2$ in place of $\sqrt{\tau_n}$ in \eqref{eq:root}.
  Our example is placed on the (very regular) metric space $\X=\R$ with the usual distance,
  with the convex but not globally differentiable potential
  \begin{align*}
    \nrg(u) = \begin{cases} u & \text{for $u\ge 0$}, \\ 0 & \text{for $u<0$}. \end{cases}
  \end{align*}
  The associated gradient flow with initial condition $u_0=1$ is the continuous curve
  \begin{align*}
    u_*(t) =
    \begin{cases}
      1-t & \text{for $0\le t\le1$}, \\
      0 &\text{for $t>1$},
    \end{cases}
  \end{align*}
  that fails to be differentiable at $t=1$.
  The solution $u_\tau^k$ to the $k$th minimization problem in \eqref{eq:BDF2scheme} is elementary to compute
  --- making a case distinction whether the minimizer is positive, negative or zero ---
  and explicitly given by
  \begin{align}
    \label{eq:recursion1}
    u_\tau^k = 
    \begin{cases}
      \frac43u_\tau^{k-1}-\frac13u_\tau^{k-2}-\frac23\tau & \text{if that expression is positive}, \\
      \frac43u_\tau^{k-1}-\frac13u_\tau^{k-2} & \text{if that expression is negative}, \\
      0 & \text{otherwise, i.e., if $0\le\frac43u_\tau^{k-1}-\frac13u_\tau^{k-2}\le\frac23\tau$}.
    \end{cases}
  \end{align}
  One easily concludes that for the initial conditions $u_\tau^0=1$ and $u_\tau^{-1}=1+\tau$,
  the $k$th approximation equals $u_\tau^k=1-k\tau$ as long as that expression is positive.
  Indeed, one has
  \begin{align*}
    \frac43u_\tau^{k-1}-\frac13u_\tau^{k-2} -\frac23\tau
    = \frac43\big(1-(k-1)\tau\big)-\frac13\big(1-(k-2)\tau\big) - \frac23\tau
    = 1-k\tau
    >0 ,
  \end{align*}
  so the first case in the recursion \eqref{eq:recursion1} applies.
  Accordingly, let $N_\tau$ be the smallest index $k\ge1$ for which $k\tau\ge1$.
  For simplicity, we assume that $u_\tau^{N_\tau}=0$, i.e., that the third case in \eqref{eq:recursion1} applies:
  \begin{align}
    \label{eq:numbertheory}
    1-N_\tau\tau = \frac43u_\tau^{N_\tau-1}-\frac13u_\tau^{N_\tau-2} -\frac23\tau \in\left[-\frac23\tau,0\right].
  \end{align}
  The other case, in which $-\tau<1-N_\tau\tau<-\frac23\tau$,
  leads to a similar result, but with more complicated formulae.
  Recalling that the two-step recursion $a_k=\frac43a_{k-1}-\frac13a_{k-2}$ has the general solution $a_k=p+3^{-k}q$
  with real parameters $p$ and $q$,
  one easily deduces from \eqref{eq:recursion1} in combination with $u_\tau^{N_\tau}=0$ 
  and $u_\tau^{N_\tau-1}=1-(N_\tau-1)\tau\in[\frac13\tau,\tau]$ because of \eqref{eq:numbertheory}
  that
  \begin{align*}
    u_\tau^k =\frac43 u_\tau^{k-1} - \frac13 u_\tau^{k-2}= -\frac12\big(1-3^{-(k-N_\tau)}\big)u_\tau^{N_\tau-1} \le -\frac16\big(1-3^{-(k-N_\tau)}\big)\tau <0
  \end{align*}
  for each index $k> N_\tau$.
  In conclusion, we have exact approximation for $t<1$, i.e.,
  \begin{align*}
    u_\tau^k = u_*(k\tau) \quad \text{for every $k$ with $k\tau<1$},
  \end{align*}
  but a residual of order $\tau$ at every point $t>1$:
  with indices $k_\tau(t)$ chosen such that $k_\tau(t)\tau\to t>1$ as $\tau\to0$,
  it follows that
  \begin{align*}
    \lim_{\tau\to0}\frac{u_*(t)-u_\tau^{k_\tau(t)}}\tau \ge \frac16\lim_{\tau\to0}\big(1-3^{-(k_\tau(t)-N_\tau)}\big)=\frac16.
  \end{align*}
  This clearly excludes the possibility of second order convergence.
\end{rmrk}


\subsection{Comparison Principle} 
The main ingredient of the proof of Theorem \ref{thm:mainthm} is the following comparison principle,
which estimates the rate at which two discrete solutions with almost identical initial data may diverge from each other.
\begin{thrm}[Comparison principle] \label{thm:CompRes}
  Let two time steps $\tau,\eta \in \left( 0 , \tau_*  \right)$ with $R:=\tau/\eta\in\N$ be given,
  and consider two pairs of initial data, $(u_\tau^{\minus 1},u_\tau^{0})$ and $(v_\eta^{\minus 1},v_\eta^{0})$.
  Let $K_0$ to $K_2$ be constants such that the Assumptions (I0)--(I2) are satisfied in both cases.
  Finally, let a time horizon $T>0$ be given.

  Then, there is a constant $C$, expressible in terms of $K_0$ to $K_2$ and $T$ alone,
  such that the piecewise constant interpolations $\bar u_\tau$ and $\bar v_\eta$ 
  of the corresponding discrete solutions $(\uk)_\kinN$ and $(v_\eta^l)_{l\in\N}$ satisfy
  \begin{align}
    \bd^2\big(\bar{u}_\tau(t) , \bar{v}_\eta(t)\big) \le C \left( \bd^2(u_\tau^0,v_\eta^0) + \tau \right)
    \label{eq:CompRes}
  \end{align}
  for all $t\in[0,T]$.
\end{thrm}
\begin{proof}
  The basic idea is to derive bounds on the distance between the discrete solutions $(u_\tau^k)_{k\in\N}$ and $(v_\eta^l)_{l\in\N}$
  at comparable times, i.e., for $(k-1)R\le l\le Rk$,
  by using the time-discrete EVI \eqref{eq:VarIneq} for each of the two solutions 
  and substituting the respective other solution for the ``observer point'' $w$.

  More specifically, 
  multiplication of \eqref{eq:VarIneq} for $u_\tau^k$ by $(4\tau)/(3+2\lambda\tau)$ yields
  \begin{equation}
    \label{eq:vi}
    \begin{split}
      &\bd^2(\uk,w) - \frac4{3+2\lambda\tau}\bd^2(\ukm,w) + \frac1{3+2\lambda\tau} \bd^2(\ukzm,w) \\
      &\leq  \frac{4\tau}{3+2\lambda\tau}\left(
        \nrg(w) - \nrg(\uk) - \frac1{\tau} \bd^2(\ukm,\uk) + \frac1{4\tau} \bd^2(\ukzm,\uk)
      \right).
    \end{split}
  \end{equation}
  For brevity, we introduce
  \begin{align*}
    g_\tau := \frac1{2+\sqrt{1-2\lambda\tau}} = \frac13+O(\tau),     \quad
    h_\tau := 2-\sqrt{1-2\lambda \tau} = 1+O(\tau), 
    \quad
    \lambda_\tau :=\frac{\log(h_\tau)}{\tau} = \lambda+O(\tau),
  \end{align*}
  where the Landau symbol $O(\tau)$ is understood for the limit $\tau \to 0$,
  and further
  \begin{align*}
    a_\tau^k(\bar{u}_\tau;w)&:=  \bd^2(\uk,w) - h_\tau^{-1} \bd^2(\ukm,w), \\
    b_\tau^k(\bar{u}_\tau;w)&:= \frac{4g_\tau}{h_\tau}                         \left( \mathcal{E}(w) - \mathcal{E}(\uk) - \frac1{\tau} \bd^2(\ukm,\uk) + \frac1{4\tau} \bd^2(\ukzm,\uk) \right).
  \end{align*}
  With these notations, the variational inequality \eqref{eq:vi} attains the following convenient form:
  \begin{align*}
    a^k_\tau(\bar{u}_\tau;w) \leq g_\tau \,  a^{k-1}_\tau(\bar{u}_\tau;w) + \tau b_\tau^k(\bar{u}_\tau;w).
  \end{align*}
  An iteration of this inequality yields
  \begin{align}
    \label{eq:chain}
    a^k_\tau(\bar{u}_\tau;w) \leq  g_\tau^k a^0_\tau(\bar{u}_\tau;w) + \tau\sum_{n=1}^k g_\tau^{k-n} \, b^n_\tau(\bar{u}_\tau;w). 
  \end{align}
  Analogously, define $g_\eta,h_\eta, \lambda_\eta$, as well as $a_\eta^l(\bar{v}_\eta;w),b_\eta^l(\bar{v}_\eta;w)$, 
  replacing $\ukzm,\ukm,\uk$ and $\tau$ with $v_\eta^{l-2},v_\eta^{l-1},v_\eta^{l}$ and $\eta$, respectively. 
  By the same argument as above, one obtains a corresponding estimate for $a^l_\eta(\bar{v}_\eta;w)$.

  Now fix a time $t \in [0,T]$, and define the three quantities 
  \[ N:=\max\{ n \mid n \tau \leq t\},\quad M:=\max\{ m \mid m\eta \leq t\}, \quad L:=M-RN.\]
  To simplify notations in the next calculations, introduce further
  \begin{align*}
    q^{k,l}:=h_\tau^kh_\eta^l\bd^2(u_\tau^k,v_\eta^l) 
    = e^{\lambda_\tau k\tau+\lambda_\eta l\eta}\bd^2\big(\bar u_\tau(k\tau),\bar v_\eta(l\eta)\big).
  \end{align*}
  The goal is to derive an estimate on the difference
  \begin{align*}
    q^{N,M} - q^{0,0} = h_\tau^N h_\eta^M \bd^2(u_\tau^N,v_\eta^{M}) - \bd^2(u_\tau^0,v_\eta^0).
  \end{align*}
  We expand this difference into telescopic sums:
  \begin{align*}
    q^{N,M}-q^{0,0} 
    &= \big(q^{N,M}-q^{N,RN}\big) + \sum_{k=1}^N\big(q^{k,Rk}-q^{k-1,R(k-1)}\big) \\
    &= \sum_{\ell=RN+1}^M(q^{N,\ell}-q^{N,\ell-1})
      + \sum_{k=1}^N\left( (q^{k,R(k-1)}-q^{k-1,R(k-1)}) + \sum_{\ell=R(k-1)+1}^{Rk}(q^{k,\ell}-q^{k,\ell-1}) \right).
  \end{align*}
  By definition of $a$ and $b$, the differences inside the sums satisfy
  \begin{align*}
    q^{k,l}-q^{k-1,l} = h_\tau^kh_\eta^la_\tau^k(\bar{u}_\tau;v_\eta^l),
    \qquad
    q^{k,l}-q^{k,l-1} = h_\tau^kh_\eta^la_\eta^l(\bar{v}_\eta;u_\tau^k).
  \end{align*}
  Insert this above and use the estimates \eqref{eq:chain} to obtain
  \begin{align}
    \nonumber
    h_\tau^N h_\eta^M \bd^2(u_\tau^N,v_\eta^{M}) - \bd^2(u_0,v_0) 
    \leq&\ I_{\tau,\eta}^{N,M}(\bar{u}_\tau,\bar{v}_\eta) \\
    \label{eq:I1}
        :=& \sum_{l=RN+1}^M h_\tau^N h_\eta^l 
          \left[ g_\eta^l a^0_\eta(\bar{v}_\eta;u_\tau^N) + \eta\sum_{n=1}^l  g_\eta^{l-n} \ b^n_\eta(\bar{v}_\eta;u_\tau^N) \right] \\
    \label{eq:I2}
        & + \sum_{k=1}^N h_\tau^k h_\eta^{R(k-1)}
          \left[ g_\tau^k a^0_\tau(\bar{u}_\tau;v_\eta^{R(k-1)}) + \tau\sum_{n=1}^k  g_\tau^{k-n} \ b^n_\tau(\bar{u}_\tau;v_\eta^{R(k-1)}) \right] 	\\
    \label{eq:I3}
        & + \sum_{k=1}^N \sum_{l=R(k-1)+1}^{Rk} h_\tau^k h_\eta^{l} 
          \left[ g_\eta^l a^0_\eta(\bar{v}_\eta;u_\tau^k) + \eta\sum_{n=1}^l g_\eta^{l-n} \ b^n_\eta(\bar{v}_\eta;u_\tau^k) \right] .
  \end{align}
  The core part of the proof of Theorem \ref{thm:CompRes} is to show that under the given hypotheses,
  \begin{align}
    \label{eq:heart}
    I_{\tau,\eta}^{N,M}(\bar{u}_\tau,\bar{v}_\eta) \le C'\tau.
  \end{align}
  The proof of \eqref{eq:heart} can be found at the end of this section. 
  In conclusion, we have 
  \begin{align*}
    e^{\lambda_\tau t+\lambda_\eta t} \bd^2(\bar{u}_\tau(t) , \bar{v}_\eta(t))  
    \leq h_\tau^N h_\eta^M \bd^2(u_\tau^N,v_\eta^{M}) 
    \leq C'\tau + \bd^2(u_0,v_0),
  \end{align*}
  which implies the inequality \eqref{eq:CompRes} with $C=e^{-2\lambda T}(1+C')$.
\end{proof}

\subsection{Proof of the Main Theorem}\label{subsec:Convergence}
With Theorem \ref{thm:CompRes} at hand, we finish the proof of Theorem \ref{thm:mainthm}. 
\begin{proof}[Proof of Theorem \ref{thm:mainthm}] 
  From Theorem \ref{thm:CompRes} it follows that $(\bar{u}_{\tau_n})_{\ninN}$ is a Cauchy family
  with respect to uniform convergence on each interval $[0,T]$.
  Indeed, since $\tau_n/\tau_m\in\N$ for arbitrary $m\ge n$ by hypothesis,
  and since the Assumptions (I0)--(I2) are satisfied with $n$-independent constants $K_0$ to $K_2$,
  there is a constant $C$ independent of $n$ and $t\in[0,T]$, 
  such that according to \eqref{eq:CompRes}, and thanks to hypothesis \eqref{eq:preroot}:
  \begin{align*}
    \bd^2(\bar{u}_{\tau_m}(t),\bar{u}_{\tau_n}(t))  
    \leq C (\bd^2(u_{\tau_m}^0,u_{\tau_n}^0) + \tau_n ) \le C(1+K_3^2)\tau_n.
  \end{align*}
  It follows that the values $(\bar u_{\tau_n}(t))_{n\in\N}$ converge in the complete metric space $(\X,\bd)$ 
  to a limit $u_*(t)$, uniformly for $t\in[0,T]$,
  and that the estimate \eqref{eq:root} holds.
  Since this argument holds for arbitrary $T>0$, the limit $u_*(t)$ is defined for all $t\ge0$.

  To prove absolute continuity of the limit curve $u_*$, we argue as usual: 
    we assign time-discrete derivatives $|u'_{\tau_n}|$ to the interpolated solutions $\bar{u}_{\tau_n}(t)$ 
  by
  \begin{align*}
    |u'_{\tau_n}|(t) := \frac{\bd(\bar{u}_{\tau_n}(t-\tau_n), \bar{u}_{\tau_n}(t))}{\tau_n}
    = \frac{\bd^2(u_{\tau_n}^{k-1},u_{\tau_n}^{k})}{\tau_n} \qquad \text{for $t \in ((k-1)\tau_n, k \tau_n ]$}.
  \end{align*}
  Thanks to the classical estimate \eqref{boundsumdk}, $|u'_{\tau_n}|$ is uniformly bounded in $L^2(0,T)$.
  Hence, $|u'_{\tau_n}|$ possesses a $L^2(0,T)$-weakly convergent subsequence (not relabelled) with limit $A \in L^2(0,T)$. 
  Choose arbitrary $s,t$ with $0\leq s \leq t \leq T$, and define $k_n(r):= \max\{k | k\tau_n \leq r\}$, 
  then 
  \begin{align*}
    \bd(\bar{u}_{\tau_n}(s),\bar{u}_{\tau_n}(t)) 
    \leq \sum_{k=k_n(s)+1}^{k_n(t)} \bd(u_{\tau_n}^{k-1},u_{\tau_n}^{k}) 
    = \int_{k_n(s) \tau_n}^{k_n(t) \tau_n} |u'_{\tau_n}| (r)\mathrm{d}r .
  \end{align*}
  In the limit $n \to \infty$, this yields
  \begin{align*}
    \bd(u_*(s),u_*(t)) = \lim_{n \to \infty} \bd(\bar{u}_{\tau_n}(s),\bar{u}_{\tau_n}(t)) 
    \leq \lim_{n \to \infty} \int_{k_{n}(s) \tau_n}^{k_{n}(t) \tau_n} |u'_{\tau_n}| 
    = \int_s^t A(r) \mathrm{d}r.
  \end{align*}
  Hence $u_* \in \mathrm{AC}^2([0,\infty),\X)$.

  It remains to prove that the limit curve $u_*$ satisfies the integrated form \eqref{eq:EVI} of the EVI. 
  Again, let $0\le s\le t\le T$, and define $k_n(r)$ be as above.
  Multiply the time-discrete EVI \eqref{eq:VarIneq} for $(u_{\tau_n}^k)_{k\in\N}$ by $\tau_n$,
  and sum from $k=k_n(s)+1$ to $k=k_n(t)$.
  On the left-hand side, we obtain after elementary manipulations:
  \begin{align*}
    J^{(1)}_n(s,t)
    &:=\tau_n\sum_{k=k_n(s)+1}^{k_n(t)}\left[
      \left(\frac3{4\tau_n}+\frac{\lambda}{2}\right) \bd^2(u_{\tau_n}^k,w)
      - \frac1{\tau_n} \bd^2(u_{\tau_n}^{k-1},w) + \frac1{4\tau_n} \bd^2(u_{\tau_n}^{k-2},w)\right] \\
    &=\frac{\lambda}2\int_{k_n(s)\tau_n}^{k_n(t)\tau_n}\bd^2(\bar u_{\tau_n}(r),w)\dd r \\
    &\qquad +\frac14\left[\big(3\bd^2(u_{\tau_n}^{k_n(t)},w)-\bd^2(u_{\tau_n}^{k_n(t)-1},w)\big) 
      - \big(3\bd^2(u_{\tau_n}^{k_n(s)},w)-\bd^2(u_{\tau_n}^{k_n(s)-1},w)\big)\right].
  \end{align*}
  Thanks to the $r$-uniform convergence of $\bar u_{\tau_n}(r)$ to $u_*(r)$,
  and since $u_*$ is continuous, we obtain in the limit
  \begin{align*}
    \lim_{n\to\infty}J^{(1)}_n(s,t) 
    = \frac\lambda2\int_s^t\bd^2\big(u_*(r),w\big)\dd r + \frac12\bd^2\big(u_*(t),w\big) - \frac12\bd^2\big(u_*(s),w\big).
  \end{align*}
  On the other hand, 
  after summation of the right-hand side of \eqref{eq:VarIneq}, 
  we estimate once again with the help of the elementary inequality \eqref{eq:binomf}
  and thus obtain
  \begin{align*}
    J^{(2)}_n(s,t)&:=\tau_n\sum_{k=k_n(s)+1}^{k_n(t)}\left[\nrg(w)-\nrg(u_{\tau_n}^k) 
      - \frac1{\tau_n} \bd^2(u_{\tau_n}^{k-1}, u_{\tau_n}^{k}) + \frac1{4\tau_n}\bd^2(u_{\tau_n}^{k-2}, u_{\tau_n}^{k})\right] \\
    &\le \int_{k_n(s)\tau_n}^{k_n(t)\tau_n}\big[\nrg(w)-\nrg\big(\bar u_{\tau_n}(r)\big)\big]\dd r
      - \frac12\bd^2(u_{\tau_n}^{k_n(t)-1},u_{\tau_n}^{k_n(t)})+ \frac12\bd^2(u_{\tau_n}^{k_n(s)-1},u_{\tau_n}^{k_n(s)}) .
  \end{align*}
  Again, thanks to local uniform convergence of $\bar u_{\tau_n}$ to the continuous limit $u_*$,
  and since $\nrg$ is lower semi-continuous thanks to Assumption (E1),
  Fatou's lemma yields that
  \begin{align*}
    \lim_{n\to\infty}J^{(2)}_n(s,t)  \le \int_s^t \big[\nrg(w)-\nrg\big(u_*(r)\big)\big]\dd r.
  \end{align*}
  Since $J^{(1)}_n(s,t) \le J^{(2)}_n(s,t)$ for all $n$ by \eqref{eq:VarIneq},
  the respective inequality follows for the limits, that is
  \begin{align*}
    \frac\lambda2\int_s^t\bd^2\big(u_*(r),w\big)\dd r + \frac12\bd^2\big(u_*(t),w\big) - \frac12\bd^2\big(u_*(s),w\big)
    \le \int_s^t \big[\nrg(w)-\nrg\big(u_*(r)\big)\big]\dd r.
  \end{align*}
  This implies the integrated EVI \eqref{eq:EVI}.
\end{proof}

\subsection{Proof of the Estimate \eqref{eq:heart}}\label{subsec:AuxCalc} 
This is a purely technical part of the convergence proof,
that uses only elementary inequalities and the classical estimates \eqref{boundsumdk}--\eqref{bounddn}.
Throughout this section, we adopt the convenient notation that $C$ is a generic constant,
which is in principle expressible in terms of the initial data $u_0$, $v_0$ and the terminal time $T$ alone,
and whose value may change from one line to the next.

To begin with, observe that since we assumed $\lambda\le0$,
we have that $g_\tau\le\frac13$ and $g_\eta\le\frac13$, and therefore
\begin{align}
 \label{eq:gseries}
 \sum_{k=0}^\infty g_\tau^k \le \frac32,
 \quad
 \sum_{l=0}^\infty g_\eta^l \le \frac32.
\end{align}
Further, we have that $h_\tau\le1$ and $h_\eta\le1$,
which means that
\begin{align}
 \label{eq:hunity}
 h_\tau^kh_\eta^l\le1
\end{align}
for arbitrary $k,l\ge0$.
On the other hand, since $h_\tau^{-1}\ge1$ and due to \eqref{eq:bound2}, it follows that
\begin{align*}
 a_\tau^0(\bar{u}_\tau;w)
 &= \bd^2(u_\tau^0,w) - h_\tau^{-1}\bd^2(u_\tau^{-1},w) \le \bd^2(u_\tau^0,w) - \bd^2(u_\tau^{-1},w) \le 2 \bd(u_\tau^{\minus 1},u_\tau^0) \bd(u_\tau^0,w)
\end{align*}
Substituting $w:=v_\eta^l$, we obtain by the triangle inequality, and thanks to estimate \eqref{bounddn},
that
\begin{align}
  \label{eq:a}
  a_\tau^0(\bar{u}_\tau;v_\eta^l)
  \le 2 \bd(u_\tau^{\minus 1},u_\tau^{0}) \big[\bd(u_*,u_\tau^0)+\bd(u_*,v_\eta^l)\big]
  \le C\bd(u_\tau^{\minus 1},u_\tau^{0})
  \le CK_2\tau,
\end{align}
where we have used that (I2) holds with constant $K_2$.
Analogously, one derives
\begin{align}
  \label{eq:b}
  a_\tau^0(\bar{v}_\eta;u_\tau^k) \le C\eta.
\end{align}
With \eqref{eq:gseries}, \eqref{eq:hunity}, \eqref{eq:a} and \eqref{eq:b} at hand,
it is now straight-forward to estimate of the terms involving $a_\tau^0$ or $a_\eta^0$.
For the expession in \eqref{eq:I1},
\begin{align*}
 \sum_{l=RN+1}^M h_\tau^Nh_\eta^lg_\eta^la^0_\eta(\bar{v}_\eta;u_\tau^N)
  \le \sum_{l=RN+1}^Mg_\eta^l C\eta
  \le\frac32C\eta.
\end{align*}
For \eqref{eq:I2},
\begin{align*}
  \sum_{k=1}^N h_\tau^kh_\eta^{R(k-1)}g_\tau^ka_\tau^0(\bar{u}_\tau;v_\eta^{R(k-1)})
  \le \sum_{k=1}^Ng_\tau^k C\tau
  \le\frac32C\tau.
\end{align*}
And finally, for \eqref{eq:I3},
\begin{align*}
  \sum_{k=1}^N\sum_{l=R(k-1)+1}^{Rk}h_\tau^kh_\eta^lg_\eta^la_\eta^0(\bar{v}_\eta;u_\tau^k)
  \le \sum_{l=1}^{RN}g_\eta^l C\eta
  \le\frac32C\eta.
\end{align*}
We turn to estimate the terms involving $b_\tau^k$ and $b_\eta^l$.
First, we use \eqref{eq:binomf} again to get a first estimate 
\begin{align*}
 b_\tau^k(\bar{u}_\tau;w) 
  \le \frac{4g_\tau}{h_\tau}\left(\nrg(w) - \nrg(u_\tau^k) +\frac1{2\tau}\bd^2(u_\tau^{k-2},u_\tau^{k-1})\right) 
  \le \frac{4g_\tau}{h_\tau}\left(\nrg(w) - \nrg(u_\tau^k) \right) +\frac2{\tau}\bd^2(u_\tau^{k-2},u_\tau^{k-1}),
\end{align*}
where we have used that $\frac{g_\tau}{h_\tau} = \frac1{3+2\lambda\tau} \le 1$.
Next, we begin by estimating the terms related to the metric.
This is done using the classical estimate \eqref{boundsumdk}:
for the expression in \eqref{eq:I1},
\begin{align*}
2\eta\sum_{l=RN+1}^Mh_\tau^Nh_\eta^l\sum_{n=1}^lg_\eta^{l-n}\frac{\bd^2(v_\eta^{n-2},v_\eta^{n-1})}{\eta}
 \le 2\eta\sum_{l=RN+1}^M\sum_{n=1}^l\frac{\bd^2(v_\eta^{n-2},v_\eta^{n-1})}{\eta} 
 \le 2R\eta\sum_{l=1}^M \frac{\bd^2(v_\eta^{l-2},v_\eta^{l-1})}{\eta}
 \le C\tau.
\end{align*}
Here we have used that $M-NR\le R$ and that $R\eta=\tau$.
For \eqref{eq:I2},
\begin{align*}
2\tau\sum_{k=1}^Nh_\tau^kh_\eta^{R(k-1)}\sum_{n=1}^kg_\tau^{k-n}\frac{\bd^2(u_\tau^{n-2},u_\tau^{n-1})}{\tau}
 \le 2 \tau \sum_{k=1}^N\left[\left(\sum_{n=k}^Ng_\tau^{n-k}\right) \frac{\bd^2(u_\tau^{k-2},u_\tau^{k-1})}{\tau}\right]
 \le \frac{3}2 C\tau.
\end{align*}
Finally, for \eqref{eq:I3},
\begin{align*}
2\eta\sum_{k=1}^N\sum_{l=R(k-1)}^{Rk}h_\tau^kh_\eta^l\sum_{n=1}^lg_\eta^{l-n}\frac{\bd^2(v_\eta^{n-2},v_\eta^{n-1})}{\eta}
 &\le 2\eta\sum_{l=0}^{RN}\sum_{n=1}^{l}g_\eta^{l-n}\frac{\bd^2(v_\eta^{n-2},v_\eta^{n-1})}{\eta} \\
 &\le 2\eta\sum_{l=1}^{RN}\left[\left(\sum_{n=l}^{RN}g_\eta^{n-l}\right) \frac{\bd^2(v_\eta^{l-2},v_\eta^{l-1})}{\eta}\right]
 \le \frac{3}2C\eta.
\end{align*}
The estimates on the expressions involving the differences of the energy values are a bit more involved.
To simplify calculations, we use that the $b$'s only contain the difference between two values of $\nrg$;
hence adding a constant to $\nrg$ does not change the $b$ values.
Consequently, since $\nrg(u_\tau^k)$ and $\nrg(v_\eta^l)$ are bounded from below thanks to \eqref{BoundEk},
we may assume without loss of generality that all $\nrg(u_\tau^k)$ and $\nrg(v_\eta^l)$ are non-negative.

The contribution of the $\nrg$ terms to \eqref{eq:I1} is immediately controlled,
recalling \eqref{eq:gseries}, \eqref{eq:hunity}, and that $M<R(N+1)$:
\begin{align*}
 \eta\sum_{l=RN+1}^Mh_\tau^Nh_\eta^l\sum_{n=1}^lg_\eta^l \frac{4g_\eta}{h_\eta}\big[\nrg(u_\tau^N)-\nrg(v_\eta^n)\big]
 \le R\eta \nrg(u_\tau^N)4 \sum_{n=0}^{M-1}g_\eta^l
 \le \frac32C\tau.
\end{align*}
Collecting all terms containing evaluations of $\nrg$ in \eqref{eq:I2} and \eqref{eq:I3} yields the following:
\begin{align*}
&\tau\sum_{k=1}^Nh_\tau^kh_\eta^{R(k-1)}\sum_{n=1}^kg_\tau^{k-n}\frac{4g_\tau}{h_\tau}\big[\nrg(v_\eta^{R(k-1)})-\nrg(u_\tau^n)\big]
+\eta\sum_{k=1}^N\sum_{l=R(k-1)+1}^{Rk}h_\tau^kh_\eta^l\sum_{n=1}^lg_\eta^{l-n} \frac{4g_\eta}{h_\eta}\big[\nrg(u_\tau^k)-\nrg(v_\eta^n)\big] \\
 =& 4\tau\sum_{k=1}^N\nrg(u_\tau^k)\left[\frac1R \frac{g_\eta}{h_\eta}\sum_{l=R(k-1)+1}^{Rk}h_\tau^k h_\eta^l\sum_{n=1}^l g_\eta^{l-n}
   - \frac{g_\tau}{h_\tau} \sum_{n=k}^N h_\tau^n h_\eta^{R(n-1)}g_\tau^{n-k}\right] \\
 &+4\tau\sum_{k=1}^N \left[\nrg(v_\eta^{R(k-1)})h_\tau^kh_\eta^{R(k-1)}\frac{g_\tau}{h_\tau}\sum_{n=1}^kg_\tau^{k-n} \right]
   - 4\eta\sum_{k=1}^{N} \sum_{l=R(k-1) +1}^{Rk} \sum_{n=1}^{l}h_\tau^k h_\eta^l g_\eta^{l-n} \frac{g_\eta}{h_\eta} \nrg(v_\eta^n).
\end{align*}
For the sum involving $\nrg(u_\tau^k)$, we obtain
\begin{align*}
I_1:=&\,\tau\sum_{k=1}^N\nrg(u_\tau^k)\,h_\tau^kh_\eta^{R(k-1)}
\left[\frac1R\frac{g_\eta}{h_\eta}\sum_{l=1}^R\left(h_\eta^l\sum_{n=0}^{R(k-1)+l-1}g_\eta^n\right)
 - \frac{g_\tau}{h_\tau} \sum_{n=0}^{N-k}(h_\tau h_\eta^Rg_\tau)^n\right] \\
 \le & \tau\sum_{k=1}^N\nrg(u_\tau^k)\,h_\tau^kh_\eta^{R(k-1)}
 \left[\frac1R \frac{1}{3h_\eta}\left(\sum_{l=1}^R\frac32\right)- g_\tau \frac{1-(h_\tau h_\eta^Rg_\tau)^{N-k+1}}{1-h_\tau h_\eta^Rg_\tau}\right] \\
 \le &\tau\sum_{k=1}^N\nrg(u_\tau^k)\,h_\tau^kh_\eta^{R(k-1)}
 \left[\frac1{2h_\eta}-\frac{g_\tau}{1-h_\tau h_\eta^Rg_\tau}+\frac{g_\tau^{N-k+1}}{1-h_\tau h_\eta^Rg_\tau}\right].
\end{align*}
Recalling that $-1\le\lambda\tau\le0$, and observing that both $g_\tau$ and $h_\tau$ are convex functions of $\tau$,
a Taylor expansion yields that
\begin{align}
  \label{eq:preelementary}
 g_\tau\ge\frac13+\frac{\lambda\tau}9\ge0,\quad
 h_\tau\ge1+\lambda\tau\ge0,
\end{align}
and similarly for $g_\eta$ and $h_\eta$.
Therefore, in combination with Bernoulli's inequality,
\begin{align*}
 h_\tau h_\eta^Rg_\tau \ge (1+\lambda\tau)(1+\lambda\eta)^R\left(\frac13+\frac{\lambda\tau}9\right)
 \ge \frac13(1+\lambda\tau)^3
 \ge \frac13\left(1+3\lambda\tau\right).
\end{align*}
Now monotonicity and convexity of the function $x\mapsto\frac1{1-x}$ lead to
\begin{align*}
 \frac1{1-h_\tau h_\eta^Rg_\tau}
 \ge \frac{1}{1-\frac13(1+3\lambda \tau)}
 = \frac32 \frac{1}{1-\frac32\lambda \tau}
 \ge \frac32 \left(1 + \frac32 \lambda \tau\right).
\end{align*}
We combine this with another application of \eqref{eq:preelementary}
and the observation that $h_\eta\ge1/(1-2\lambda\eta)$ thanks to \eqref{eq:hypos}
to arrive at:
\begin{align}
  \label{eq:elementary}
  \frac1{2h_\eta} - \frac{g_\tau}{1-h_\tau h_\eta^Rg_\tau}
  \le \frac{1}{2} (1 -2\lambda \eta) -\frac13\left(1+\frac13\lambda \tau\right)\frac32\left( 1 + \frac32 \lambda \tau\right)
  \le \frac12- \lambda \eta- \frac12 - \frac{11}{6}\lambda \tau \le - 3 \lambda \tau.
\end{align}
This yields, in combination with the bound \eqref{BoundEk} on $\nrg$,
\begin{align*}
 I_1
 &\le \tau\sum_{k=1}^N\nrg(u_\tau^k)\,h_\tau^kh_\eta^{R(k-1)}\left[\frac{3g_\tau^{N-k+1}}{1-h_\tau h_\eta^Rg_\tau} -3\lambda\tau\right]\\
 &\le \tau C\sum_{k=1}^N\left[\frac92g_\tau^{N-k+1}-3\lambda\tau\right]
 \le \tau C\left(\frac{27}{4}-3\lambda\tau N \right)
 \le C[1+(-\lambda)T]\ \tau.
\end{align*}
We turn to the sums involving values of the form $\nrg(v_\eta^l)$, i.e.,
\begin{align*}
 I_2:=
        \tau\sum_{k=1}^N \left[\nrg(v_\eta^{R(k-1)})h_\tau^kh_\eta^{R(k-1)} \frac{g_\tau}{h_\tau}\sum_{n=1}^kg_\tau^{k-n} \right]
   - \eta\sum_{k=1}^{N} \sum_{l=R(k-1) +1}^{Rk} \sum_{n=1}^{l}h_\tau^k h_\eta^l g_\eta^{l-n} \frac{g_\eta}{h_\eta} \nrg(v_\eta^n) .
\end{align*}
In order to join these two sums into a single one --- similar to the sum involving $\nrg(u_\tau^k)$ above ---
we are going to apply a small shift to the indices inside $\nrg(v_\eta^{R(k-1)})$.
To this end, observe that an iteration of the energy estimate \eqref{eq:EnergyDim} yields
\begin{align*}
\nrg(v_\eta^{Rk})\le\nrg(v_\eta^l)+\frac1{4\eta}\bd^2(v_\eta^{Rk-1},v_\eta^{Rk}) 
\end{align*}
as soon as $0\le l\le Rk$.
Further, for such $k$ and $l$, we have $h_\tau^k\le h_\tau^{l/R}$ since $h_\tau\le1$.
This allows us to estimate the first sum in $I_2$ as follows:
\begin{align*}
 \tau\sum_{k=1}^N \left[\nrg(v_\eta^{R(k-1)})h_\tau^kh_\eta^{R(k-1)} \frac{g_\tau}{h_\tau} \sum_{n=1}^kg_\tau^{k-n} \right]
 & \leq \tau \nrg(v^0) + \frac{\tau}{2h_\tau} \sum_{k=1}^{N} h_\tau^{k+1} h_\eta^{Rk} \nrg(v_\eta^{Rk})  \\
 & \leq \tau C + \frac{\tau}{2Rh_\tau} \sum_{k=1}^N \sum_{l=R(k-1)+1}^{Rk} h_\tau^{l/R+1} h_\eta^l\left[\nrg(v_\eta^l) + \frac1{4\eta} \bd^2(v_\eta^{Rk-1},v_\eta^{Rk}) \right] \\
 & \leq \tau C + \frac{\eta}{2h_\tau} \sum_{l=1}^{RN} h_\tau^{l/R+1} h_\eta^l\nrg(v_\eta^l) +\frac{\tau}{4 h_\tau} \sum_{k=1}^N  \frac{\bd^2(v_\eta^{Rk-1},v_\eta^{Rk})}{2\eta} \\
 &\le \tau C + \frac{\eta}{2h_\tau} \sum_{l=1}^{RN} h_\tau^{l/R+1} h_\eta^l\nrg(v_\eta^l),
\end{align*}
where we have used the classical estimate \eqref{boundsumdk} and a lower bound for $h_\tau$ in the last inequality.
The second sum in $I_2$ is estimated as follows, using that $h^k_\tau \geq h_\tau^{l/R+1}$ for $R(k-1)<l\le Rk$:
\begin{align*}
 \eta\sum_{k=1}^{N} \sum_{l=R(k-1) +1}^{Rk} \sum_{n=1}^{l}h_\tau^k h_\eta^l g_\eta^{l-n} \frac{g_\eta}{h_\eta}\nrg(v_\eta^n)
 & \geq \eta \sum_{l=1}^{RN} \sum_{n=1}^{l} h_\tau^{l/R+1} h_\eta^{l} g_\eta^{l-n+1} \nrg(v_\eta^n)\\
 & = \eta g_\eta \sum_{l=1}^{RN} \nrg(v_\eta^l) h_\tau^{l/R+1} h_\eta^{l} \sum_{n=0}^{RN-l} \big(h_\tau^{1/R} h_\eta g_\eta\big)^n \\
 & = \eta g_\eta \sum_{l=1}^{RN} \nrg(v_\eta^l) h_\tau^{l/R+1} h_\eta^{l} \frac{1- (h_\tau^{1/R} h_\eta g_\eta)^{RN-l+1}}{1-h_\tau^{1/R} h_\eta g_\eta}.
\end{align*}
Substituting these estimates into the expression for $I_2$ yields a single sum,
\begin{align*}
 I_2 \le \tau C
 + \eta \sum_{l=1}^{RN} h_\tau^{l/R+1} h_\eta^l\nrg(v_\eta^l)
 \left[\frac1{2h_\tau} - \frac{g_\eta}{1-h_\tau^{1/R} h_\eta g_\eta} + \frac{g_\eta^{RN-l+1}}{1-h_\tau^{1/R} h_\eta g_\eta}\right].
\end{align*}
Arguing similarly as in the derivation of \eqref{eq:elementary},
we estimate
\begin{align*}
 h_\tau^{1/R}h_\eta g_\eta
 \geq \left(1+\lambda \frac{\tau}{R}\right)(1+\lambda \eta)\left(\frac13+\frac{\lambda \eta}9\right)
 \ge \frac{1}{3} (1+\lambda \eta)^3 \geq \frac{1}{3} (1+3 \lambda \eta),
\end{align*}
and consequently,
\begin{align*}
\frac1{2h_\tau} - \frac{g_\eta}{1-h_\tau^{1/R}h_\eta g_\eta} & \le \frac12 (1- 2 \lambda \tau) - \frac13 (1+\frac13\lambda \eta)\frac{3}{2}(1+\frac32\lambda \eta) \le \frac12 -  \lambda \tau - \frac12 - \frac{11}{6}\lambda \eta \le -3 \lambda \tau.
\end{align*}
In conclusion, we obtain with the help of \eqref{BoundEk} that
\begin{align*}
 I_2 \le C\tau + \eta\sum_{l=1}^{RN}\nrg(v_\eta^l)\left[\frac32g_\eta^{RN+1-l}+3(-\lambda)\tau\right] 
 \le C\tau + \frac{9C}4\eta + (-\lambda)3\tau \eta  RN
 \le C\tau + \frac{9C}4\eta + 3(-\lambda)T \tau.
\end{align*}
Collecting all terms, we finally obtain the desired estimate \eqref{eq:heart}.


\section{Illustration by Numerical Experiments}\label{sec:Num}
In this section, we illustrate the convergence of our variational BDF2 method in comparison to the implicit Euler scheme in several numerical experiments.
As examples, we have chosen 
a flow on the two-dimensional sphere $\mathbb S^2$, 
a reaction-diffusion equation as flow on the Hilbert space $L^2([0,1])$,
and an aggregation-diffusion equation as flow on the space $\mathcal P([-1,1])$ of probability measures on $[-1,1]$, 
equipped with the the $L^2$-Wasserstein distance $\wass$.
We observe that the order of convergence is indeed very close to two in each of our simulations.
This underlines our philosophy that one reaches the optimal order in ``typical'' problems,
despite the fact that our main Theorem \ref{thm:mainthm} only provides order one-half, 
and that there are specific counter-examples with sub-optimal converge rates, 
like in Remark \ref{rmk:counter}. \\

In each of the examples below, 
we compare the numerical results for the implicit Euler scheme and for the BDF2 method at various moderately small time steps $\tau>0$ 
to a reference solution that is obtained by the BDF2 method with a very small time step $\tau_\text{ref}$.
The approximation with the implicit Euler method of step size $\tau>0$ 
--- see Section \ref{subsec:MMS} for details --- 
is denoted by $\bar{u}_\tau^{(1)}$, and the approximation with BDF2 by $\bar{u}_\tau^{(2)}$, respectively. 
For the time-discrete initial data, we choose the original datum $u^0$ for both schemes at $t=0$, 
and  for the second initial datum (at $t=\tau$) of the BDF2 method,
we use the result of the first step of the implicit Euler scheme. This choice ensures in the ODE setting the enhanced convergence rate of order two, since the startup calculation with one step of the implicit Euler scheme is of order two \cite[Theorem 7.23]{BD}. \\

The numerical rate of convergence is then computed as follows.
In addition to the very small reference time step $\tau_\text{ref}$, we choose a moderately large time step $\tau_\text{coarse}$ that is an integer multiple of $\tau_\text{ref}$.
Then, we calculate  $\bar{u}_\tau^{(1)}$ and  $\bar{u}_\tau^{(2)}$ for several intermediate time steps $\tau\in(\tau_\text{ref},\tau_\text{coarse})$ 
that are chosen such that $\tau$ is an integer multiple of $\tau_\text{ref}$, 
and $\tau_\text{coarse}$ is an integer multiple of $\tau$.
For each such choice of $\tau$, 
the respective solutions $\bar{u}_\tau^{(1)}$ and  $\bar{u}_\tau^{(2)}$ are compared to the reference solution $\bar{u}_\text{ref}^{(2)}$:
specifically, we calculate a mean numerical error 
by taking the average of the distances $\bd(\bar{u}_\tau^{(i)}(t_k),u_\text{ref}(t_k))$ at times $t_k=k\tau_\text{coarse}\in[0,T]$ on the coarsest grid. \\

All simulations have been performed with \texttt{MATLAB}.
Both variational schemes are implemented by solving the sequence of variational problems  
using the built-in function \texttt{fmincon}.

\subsection{Gradient Flow on the Sphere $\mathbb{S}^2$}\label{subsec:Sphere}
The first test problem is placed on the unit 2-sphere $\mathbb{S}^2$ equipped with the intrinsic (great-circle) distance $\bd_{\mathbb{S}^2}$, 
\begin{align*}
 \X:=\mathbb{S}^2=\{u\in\R^3\,|\,u_1^2+u_2^2+u_3^2=1\}\subset\R^3, \qquad  \bd_{\mathbb{S}^2}(u,v)=\arccos(u_1v_1+u_2v_2+u_3v_3).
\end{align*}
For the potential $\nrg:\mathbb{S}^2\to\R$, we choose the restriction of
\begin{align*}
  \tilde\nrg(u)= \sum_{i=1}^3 \big(u_i -\frac12\big)\big(u_i+\frac12\big)^2.
\end{align*}
The corresponding gradient flow satisfies the ODE
\[\dot u=-\nabla_{\mathbb{S}^2}\nrg(u) = \Pi_u\big[-\nabla\tilde\nrg(u)\big],\]
where $\Pi_u[v]=v-u^Tv$ is the projection of a vector $v$ to the tangent space of $\mathbb{S}^2$ at $u$.
Its flow lines are sketched in Figure \ref{fig:sphere} (left).
The example falls into the class of gradient flows on Riemannian manifolds that is covered by Theorem \ref{thm:diffgeo}.

A series of simulations has been performed for the initial datum
\begin{align*}
  u^0 = \frac1{\sqrt{30}}  (1,2,5)
\end{align*}
and the reference step size $\tau_{ref}=10^{-5}$.
The observed numerical convergence rates are 1.00 for the implicit Euler method, and 2.06 for BDF2,
see Figure \ref{fig:sphere} (right).
Further experiments with different initial data and other potentials yield very similar results.
In this smooth, finite dimensional setting, second order convergence of the BDF2 method was naturally expected.

\begin{figure}[h]
  \label{fig:sphere}
  \includegraphics[width=0.5\textwidth]{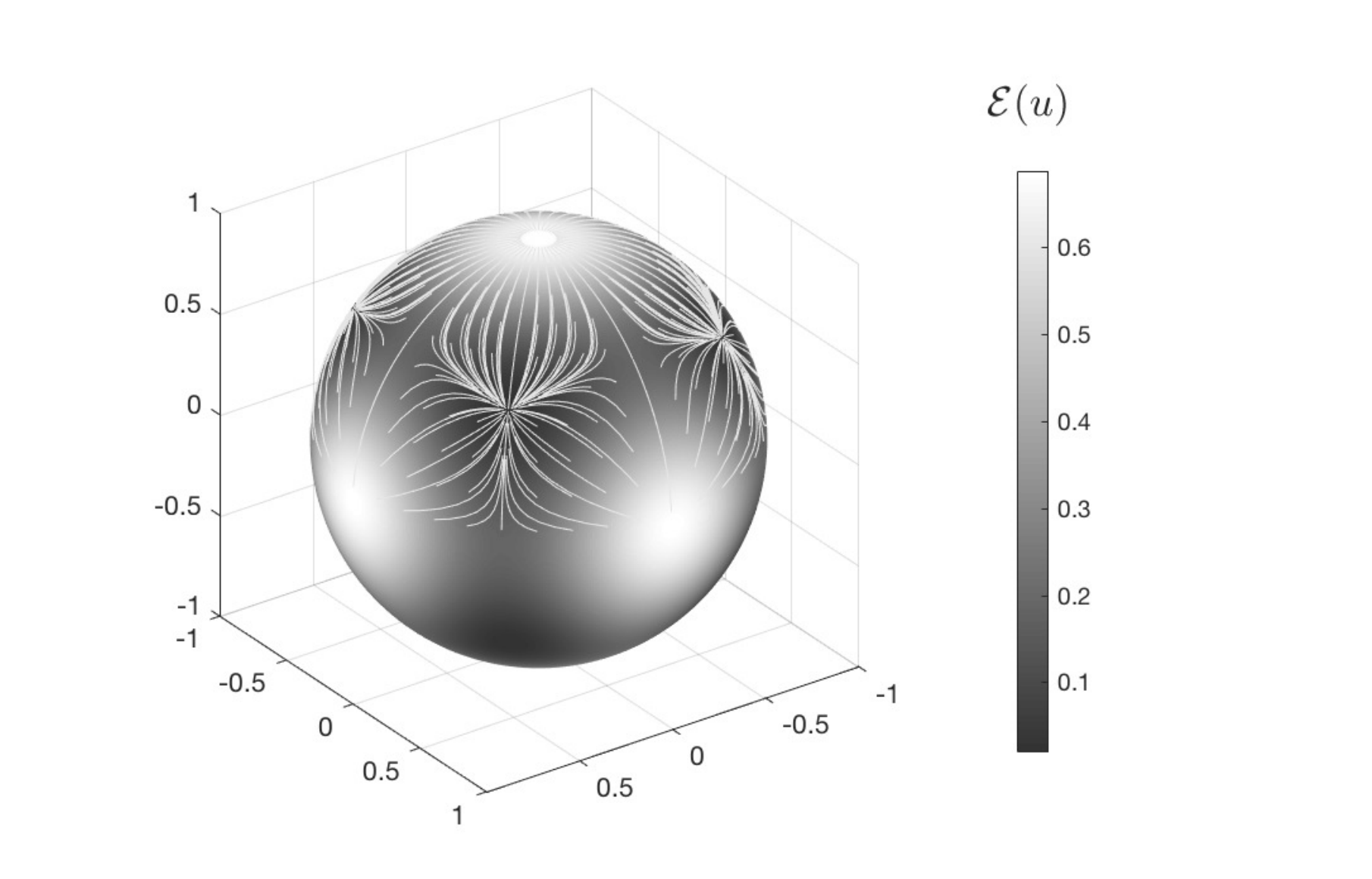}%
\includegraphics[width=0.5\textwidth]{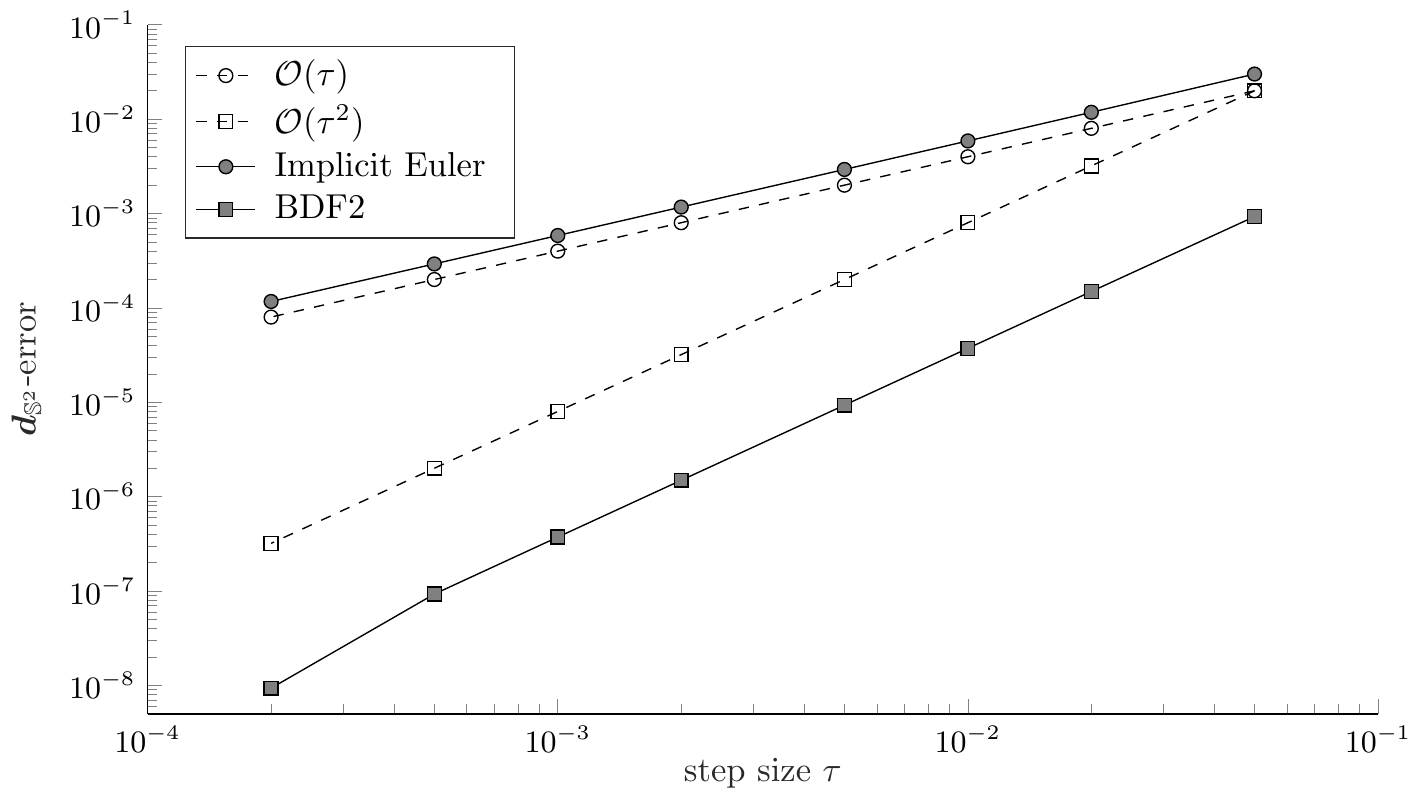}
\caption{\emph{Gradient flow on the Sphere $\mathbb{S}^2$}.
  Left: the values of $\nrg$ are color-coded by gray scale.
  The white lines are sample trajectories of the gradient flow generated by $\nrg$. 
  Right: the $\bd_{\mathbb{S}^2}$-error plot of $\bar{u}_\tau^{(i)}$ compared  with  $\bar{u}_{\tau_{ref}}$.}
\end{figure}

\subsection{Reaction-Diffusion Equation with Obstacle}\label{subsec:ReaDif} 
Next, we consider the constrained reaction-diffusion equation 
\begin{align*}
  \partial_t u&= \Delta u + 60 u^3 \quad
  \text{subject to } \quad   \left|u\right|\leq 1 
\end{align*}
on $\Omega=\left[0,1\right]$. 
This PDE constitutes a gradient flow on the Hilbert space $L^2([0,1])$ for the energy 
\begin{align*}
  \nrg(u)= \begin{cases} 
    \frac12\int_0^1 \big(\partial_xu(x) \big)^2 \mathrm{d} x - 15 \int_0^1 u(x)^4 \mathrm{d} x, 
    & \mathrm{for} \ u\in H^1([0,1]),\ \left|u\right|\leq 1, \\
    + \infty, & \mathrm{otherwise}. \end{cases} 
\end{align*}
The second variation of $\nrg$ amounts to
\begin{align*}
  \mathrm{D}^2\nrg(u)[\varphi]^2 
  = \int_0^1 \big(\partial_x\varphi(x) \big)^2 \mathrm{d} x - 180 \int_0^1 u(x)^2\varphi(x)^2 \mathrm{d} x
  \ge -180 \|\varphi\|^2_{L^2},
\end{align*}
since $0\le u(x)^2\le1$.
Hence $\nrg$ is uniformly semi-convex of modulus $\lambda=-90$. \\

For numerical approximation, we first perform the implicit Euler or BDF2 method for discretization in time,
then we apply a spatial discretization of the PDE, using central finite differences.
The qualitative behavior of the approximate solution for the initial condition
\begin{align*}
  u_0(x)  = \frac12 \sin(2 \pi x) + \frac14
\end{align*}
has been plotted in Figure \ref{fig:obstacle} (left).
Notice that the upper barrier is hit after a short transient time. 
The reference step size is $\tau_\text{ref}= 10^{\minus 6}$.
Since we are interested in the convergence rate of the temporal discretization for the PDE,
we need to estimate the influence of the additional spatial discretization on the numerical error.
For that reason, the experiment is carried out with different choices of the spatial resolution,
using $K=50, 100, 250, 500, 1000$ grid points.\\

Our results on the numerical error are given in Figure \ref{fig:obstacle} (right).
The error curve for the implicit Euler scheme is proportional to $\tau$, as expected.
For time steps $\tau>10^{\minus 5}$, the error curve for the BDF2 scheme is almost perfectly proportional to $\tau^2$,
and there is no significant dependence on the spatial discretization.
For very small steps $\tau\le10^{\minus 5}$, 
there is apparently an additional contribution to the numerical error due to the spatial discretization,
however as $K$ is increased, the error curve extends its approximate proportionality to $\tau^2$ also into that regime.
This is a strong indicator that for a purely temporal discretization by BDF2, the order of convergence is indeed quadratic in $\tau$. 
We performed further experiments with different initial data, and with variants of the energy functional. 
The results remain approximately the same.
\begin{figure}[h]
  \includegraphics[width=0.5\textwidth]{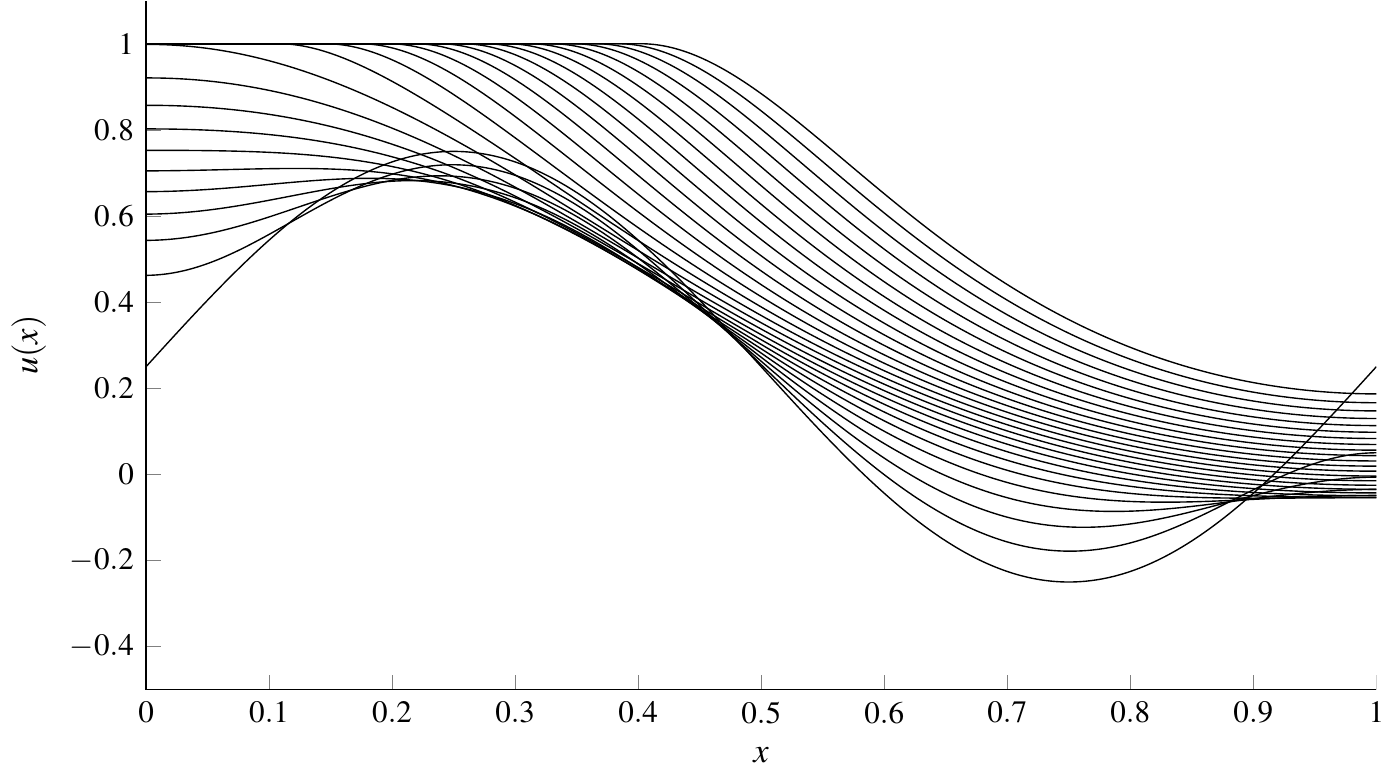}%
  \includegraphics[width=0.5\textwidth]{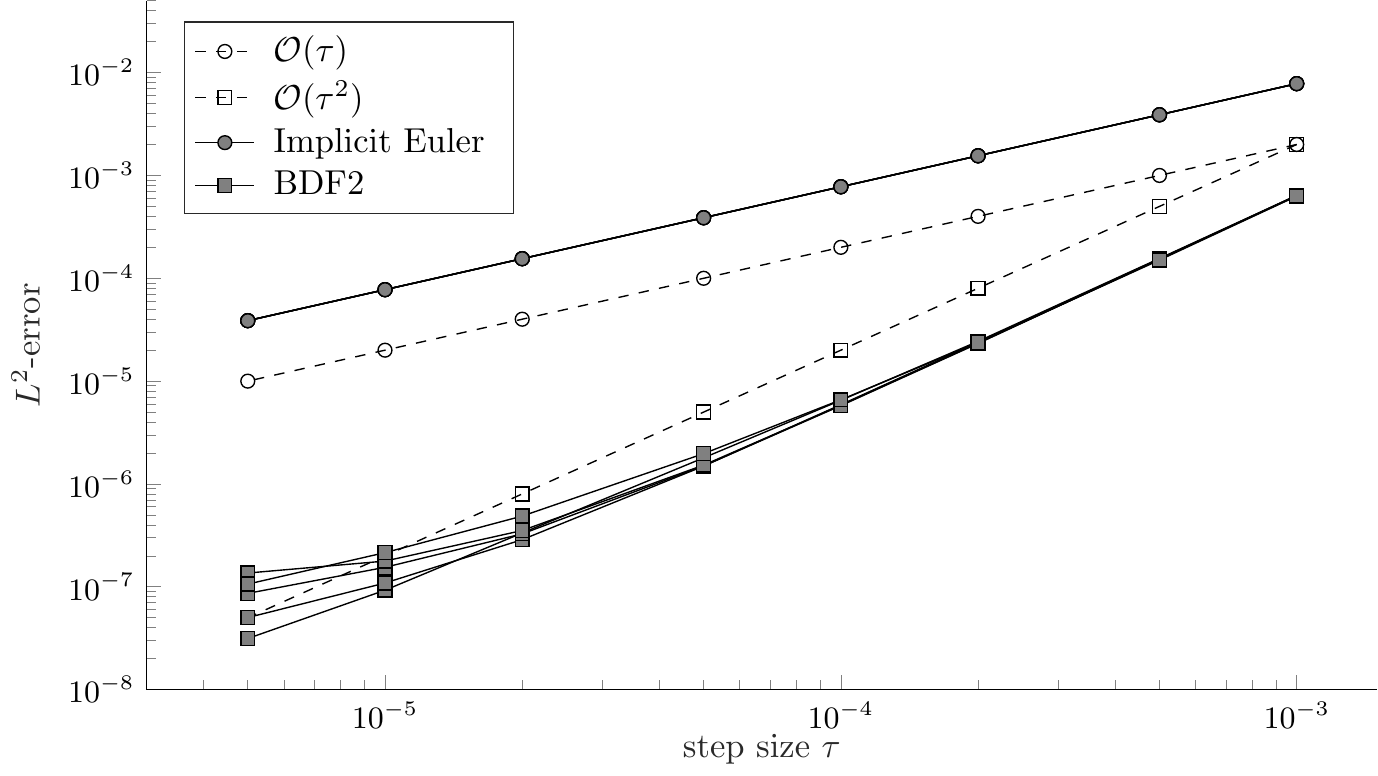}
  \caption{{\emph{Reaction-diffusion equation with obstacle}}: Evolution of the reference solution $\bar{u}_{\tau_{ref}}$ (left). The $L^2$-error plot of $\bar{u}_\tau^{(i)}$ compared with $\bar{u}_{\tau_{ref}}$ for different $K$ (right).}
  \label{fig:obstacle}
\end{figure}

\subsection{Aggregation-Diffusion Equations}\label{subsec:Aggeq}
In our last example,
we study discretizations of the following aggregation-diffusion equation on $\Omega=[-1,1]$:
\begin{align}
  \label{eq:aggr}
  \partial_t u = \Delta u + \partial_x(u W'\ast u).
\end{align}
For the interaction kernel, we use $W(x)= 2x^4-x^2$. 
Weak solutions to \eqref{eq:aggr} conserve mass and positivity, 
so we may restrict attention to solutions $u$ that are probability densities.
Under this restriction, solutions to \eqref{eq:aggr} correspond to the gradient flow 
on the space $\X=\mathcal{P}([-1,1]) $ of probability measures $\mu$ 
with respect to the $L^2$-Wasserstein distance $\bd=\wass$ 
for the energy functional 
\begin{align*}
  \nrg(\mu) := \begin{cases} 
    \int_\Omega u(x) \log(u(x)) \mathrm{d} x + \frac12 \iint_{\Omega\times\Omega} u(x)W(x-y)u(y) \mathrm{d} y \mathrm{d} x, 
    & \text{if $\mathrm\mu(x)=u(x)\mathrm dx$ with density $u\in L^1(\Omega)$}, \\
    + \infty, 
    & \mathrm{otherwise}. 
  \end{cases}
\end{align*}
For numerical simulation, we employ the isometry of the the Wasserstein space $(\mathcal{P}(\Omega),\wass)$
and the space $\tilde\X$ of non-decreasing c\`{a}dl\`{a}g functions $X:[0,1]\to\Omega$,  equipped with the $L^2([0,1])$-norm. 
This isometry is realized by assigning to each $\mu$ its inverse distribution function $X_\mu$, i.e.,
$X_\mu$ is the unique function in $\tilde\X$ with
\[ \xi = \int_0^{X_\mu(\xi)} 1 \ \mathrm d \mu(x) \quad \text{for all $\xi\in[0,1]$}. \]
Accordingly, the Wasserstein gradient flow transforms into an $L^2([0,1])$-gradient flow on $\tilde\X$ 
with the energy functional 
\begin{align*}
  \tilde\nrg(X):= 
  -\int_{[0,1]} \log(\partial_{\xi} X(\xi)) \mathrm{d} \xi + \frac12 \iint_{[0,1] \times [0,1]} W(X(\xi)-X(\eta)) \mathrm{d} \xi \mathrm{d} \eta.
\end{align*}
In the numerical experiments,
we prescribe an initial datum $u_0$ via its inverse distribution function
\begin{align*}
  X_{u_0}(\xi):= 2\xi-1+ \frac{1}{8\pi} \sin(8\pi\xi) \cdot \left( 10(\xi(\xi-0.5)(x-1)) +1 \right). 
\end{align*}
Concerning the discretization in space, we proceed as in the previous example,
using central finite differences on the space $\tilde\X$ with $K=50, 100, 250, 500, 1000$ spatial grid points. 
The qualitative behavior of the reference solution (in original variables with $\tau_{ref}=10^{\minus 6}$, and $K=1000$) 
is sketched in Figure \ref{fig:aggr} (left). \\
 
Our results on the numerical error are given in Figure \ref{fig:aggr} (right).
The error curves for the implicit Euler and the BDF2 schemes, respectively, 
are almost perfectly proportional to $\tau$ and $\tau^2$.
The results are comparable to (and even better than in) the previous example; 
we do not observe any significant effect of the spatial discretization, even for very small time steps.
This indicates that the purely temporal discretization of the original PDE with BDF2 leads an approximation error $\tau^2$.
\begin{figure}[h]
  \includegraphics[width=0.5\textwidth]{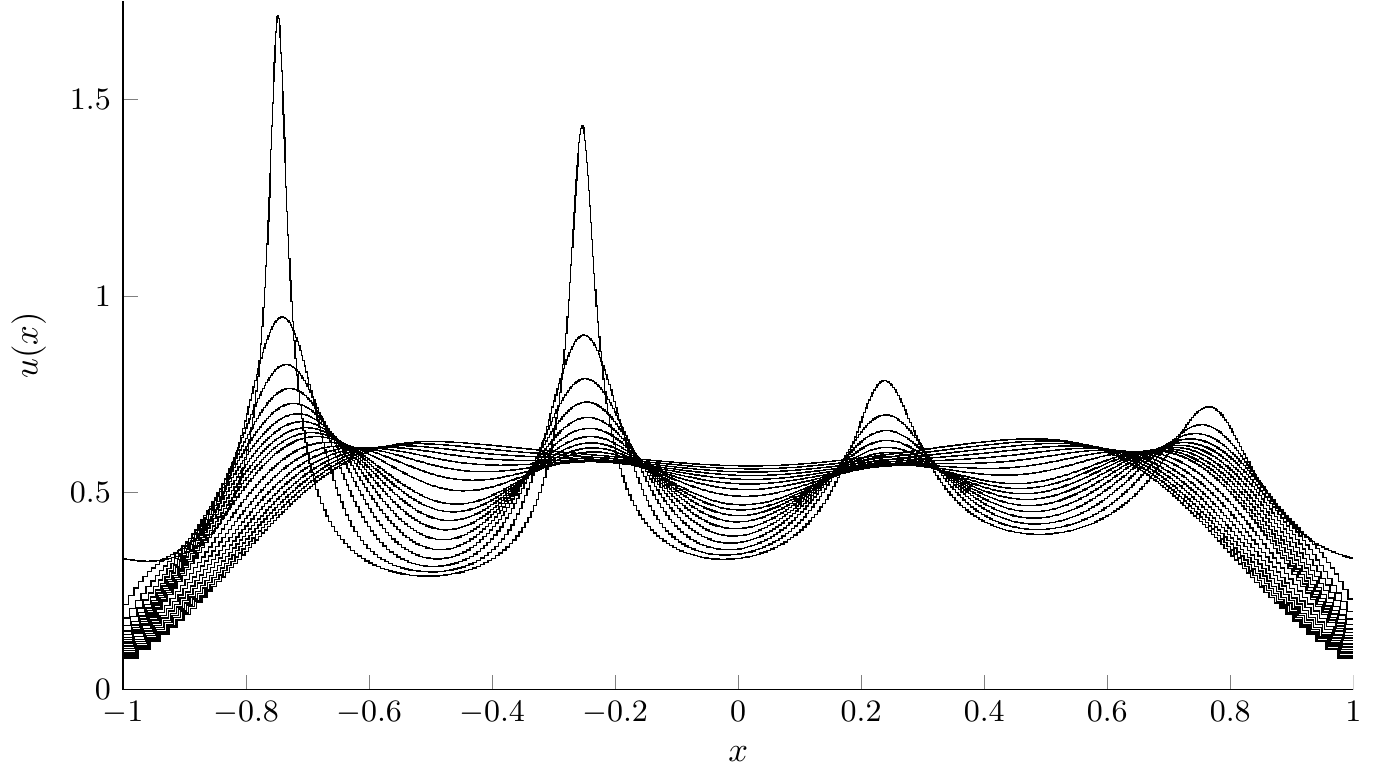}%
  \includegraphics[width=0.5\textwidth]{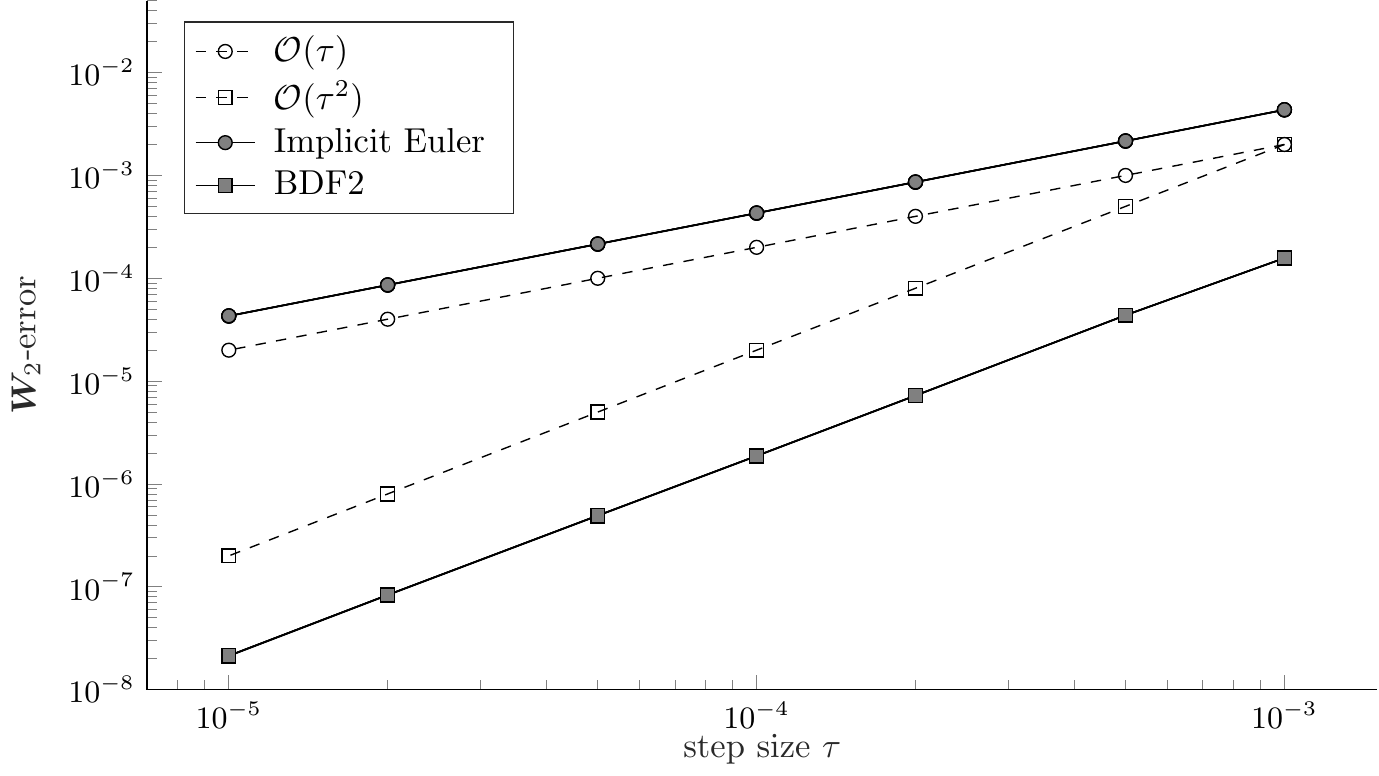}
  \label{fig:aggr}
  \caption{{\emph{Aggregation-diffusion equation}}: Evolution of the reference solution $\bar{u}_{\tau_{ref}}$ (left) and  the $\wass$-error plot of $\bar{u}_\tau^{(i)}$ compared  with $\bar{u}_{\tau_{ref}}$ for different $K$ (right).}
\end{figure}


\appendix

\section{Cross curvature}
\label{apx:cross}
In this appendix, we define the term \emph{cross curvature} that has been used in hypothesis (KM3) in Section \ref{sct:diffgeo},
and we briefly discuss its significance in the proof of \cite[Corollary 2.11]{kim2012towards},
which we have been used for showing Theorem \ref{thm:diffgeo}.
Starting from the Riemannian manifold $(\mf,\g)$, 
one defines a further manifold $\smf$ with indefinite metric tensor $\sg$ as follows:
\begin{align*}
  \smf = \{ (x,y)\,|\,x,y\in\mf,\,x\notin\cut(y),\,y\notin\cut(x)\},
  \quad
  \sg_{(x,y)}[(\xi,\eta)]^2 = -\frac12\frac{d}{ds}\bigg|_{s=0}\,\frac{d}{dt}\bigg|_{t=0}\,\bd^2\big(\exp_x(s\xi),\exp_y(t\eta)\big).
\end{align*}
The cross-curvature of $(\mf,\g)$ at the point pair $(x,y)\in\smf$ is defined 
as the tensor $\scurv^{(\smf,\sg)}_{(x,y)}$ of sectional curvatures in $(\smf,\sg)$ at $(x,y)$.
More explicitly, for $\hat z=(\hat v,\hat w),\check z=(\check v,\check w)\in T_{(x,y)}\smf$,
\begin{align*}
  \scurv^{(\smf,\sg)}_{(x,y)}[\hat z,\check z]^2 = \sum_{i,j,k,l} R_{i,j,k,\ell}\hat z^i\check z^j\hat z^k\check z^\ell,
\end{align*}
with $R$ the Riemann curvature tensor of $(\smf,\sg)$ at $(x,y)$.
One says that $(\mf,\g)$ has non-negative cross curvature if, 
for each $(x,y)\in\smf$ and $v\in T_x\mf$, $w\in T_y\mf$,
\begin{align}
  \label{eq:nncross}
  \scurv^{(\smf,\sg)}_{(x,y)}[(v,0),(0,w)]^2\ge0.
\end{align}
The role of $\scurv^{(\smf,\sg)}$ is probably best understood from the following alternative representation 
that has been derived in \cite[Lemma 4.5]{KMold}:
given a curve  $y:(-\epsilon,\epsilon)\to\mf$ and points $x_0,x_1\notin\cut(y(0))$,
then
\begin{align}
  \label{eq:mtw}
  -\frac{d^2}{dt^2}\Big|_{t=0}\,\frac{d^2}{ds^2}\Big|_{s=0}\,\bd^2\big([x_0,x_1;y_0]_t,y(s)\big)
  = \scurv^{(\smf,\sg)}_{(x_0,y_0)}[(v,0),(0,w)]^2,
\end{align}
where $y_0:=y(0)$, $w:=\dot y(0)\in T_{y_0}\mf$, and $v\in T_{x_0}\mf$ the derivative of $t\mapsto[x_0,x_1;y_0]_t$ at $t=0$.
The crucial point here is that the fourth order mixed derivative on the left-hande side only depends on $v$ and $w$,
without any requirement on $y$'s second derivative, like $y$ being a segment.

The representation \eqref{eq:mtw} is also the key ingredient to the proof of \cite[Corollary 2.11]{kim2012towards},
which we have built upon in Section \ref{sct:diffgeo}.
The cited corollary states the following:
under the hypotheses (KM0)--(KM3) given in Section \ref{sct:diffgeo} 
--- including non-negative cross-curvature ---
the function defined in \eqref{eq:helpf} is convex, for any choice of $u,v,\gamma_0,\gamma_1\in\mf$.
To prove this,
Kim and McCann consider for each $s\in[0,1]$ the segment $z^{(s)}_{(\cdot)}:=[u,v;\gamma_s]_{(\cdot)}$ connecting $u$ to $v$ with base $\gamma_s$,
and define the extended function $F:[0,1]\times[0,1]\to\R$ with
\begin{align}
  \label{eq:mccann}
  F(s,t) = \bd^2(u,\gamma_s) - \bd^2(z^{(s)}_t,\gamma_s).
\end{align}
In these notations, the claim is that $\partial_s^2F(s,1)\ge0$ for all $s\in[0,1]$.
By definition of $F$, it is clear that $\partial_s^2F(s,0)=0$.
A direct calculation in local coordinates shows that $\partial_t\partial_s^2F(s,0)=0$ as well.
Finally, thanks to non-negative cross-curvature, 
it follows from \eqref{eq:mtw} that $\partial_t^2\partial_s^2F(s,t)\ge0$, for each $s,t\in[0,1]$.
Notice that $\gamma$ plays the role of the ``general'' curve $y$, whereas $z$ is of the required exponential form.
This concludes the proof.
\begin{rmrk}
  It turns out that $\scurv^{(\smf,\sg)}$ is identical to the celebrated \emph{MTW tensor},
  which is usually defined via the expression on the left-hand side of \eqref{eq:mtw}.
  The weak regularity condition of Ma, Trudinger and Wang \cite{ma2005regularity}
  is \eqref{eq:nncross}, but only for those $v,w\in T_{x(0)}\mf$ with $\g_{x(0)}[v,w]=0$.
  It constitutes a necessary condition for the regularity of optimal transport maps on $\mf$ with cost function $c(x,y)=\frac12\bd^2(x,y)$.  
\end{rmrk}


\bibliographystyle{abbrv}
   \bibliography{thlit}

\end{document}